\pdfoutput=1 

\documentclass[12pt]{article}
\usepackage[utf8]{inputenc}
\usepackage{amsmath}
\usepackage{amssymb}
\usepackage{amsthm}
\usepackage{tikz}
\usepackage{setspace}
\usetikzlibrary{patterns}

\newtheorem{theorem}{Theorem}[section]
\newtheorem{definitio}[theorem]{Definition}
\newenvironment{definition}{\begin{definitio} \rm }{\end{definitio}}
\newtheorem{lemma}[theorem]{Lemma}
\newtheorem{proposition}[theorem]{Proposition}
\newtheorem{corollary}[theorem]{Corollary}
\theoremstyle{remark}
\newtheorem*{remark}{Remark}
\newtheorem*{notation}{Notation}
\newtheorem{fact}[theorem]{Fact}

\newenvironment{partext}[1]{\begin{minipage}{#1} \begin{center}}{\end{center}\end{minipage}}

\def\fl#1{\smash{\mathop{\hbox to 11mm{ \rightarrowfill\ }}\limits^{\scriptstyle{#1}}}}

\newcommand{\N}{\mathbb{N}}
\newcommand{\Z}{\mathbb{Z}}
\newcommand{\Q}{\mathbb{Q}}
\newcommand{\K}{\mathbb{K}}

\newcommand{\Qt}{\mathbb{Q}(t)}
\newcommand{\Ztt}{\mathbb{Z}[t^{\pm1}]}
\newcommand{\Qtt}{\mathbb{Q}[t^{\pm1}]}
\newcommand{\iso}{\fl{\scriptstyle{\cong}}}
\newcommand{\cub}{C^{g^+}_{g^-}}
\newcommand{\tcob}{\widetilde{\mathcal{C}ob}}
\newcommand{\cob}{\mathcal{C}ob}
\newcommand{\tlcob}{\widetilde{\mathcal{LC}ob}}
\newcommand{\lcob}{\mathcal{LC}ob}
\newcommand{\btt}{{}_b^t\mathcal{T}}
\newcommand{\tbtt}{{}_b^t\widetilde{\mathcal{T}}}
\newcommand{\tang}{\mathcal{T}_q}
\newcommand{\ttang}{\widetilde{\mathcal{T}}_q}
\newcommand{\tcub}{\mathcal{T}_q\mathcal{C}ub}
\newcommand{\ttcub}{\widetilde{\mathcal{T}}_q\mathcal{C}ub}
\newcommand{\A}{\mathcal{A}}
\newcommand{\tA}{\widetilde{\mathcal{A}}}
\newcommand{\bA}{\bar{\mathcal{A}}}
\newcommand{\tAw}{\widetilde{\mathcal{A}}^\textrm{\textnormal{w}}}
\newcommand{\tAts}{{}^{ts}\hspace{-4pt}\widetilde{\mathcal{A}}}
\newcommand{\tAtsm}[2]{{}^{ts}\hspace{-4pt}\widetilde{\mathcal{A}}(*_{\lfloor #1\rceil^+\cup\lfloor #2\rceil^-})}
\newcommand{\pgf}[2]{*_{\lfloor #1\rceil^+\cup\lfloor #2\rceil^-}}
\newcommand{\holw}{Hol$_\textrm{w}$}
\newcommand{\xd}{\,\raisebox{-0.5ex}{\begin{tikzpicture} \draw[->] (0,0) -- (0,0.4);\end{tikzpicture}}\,}
\newcommand{\xdop}{\,\raisebox{-0.5ex}{\begin{tikzpicture} \draw[<-] (0,0) -- (0,0.4);\end{tikzpicture}}\,}
\newcommand{\xc}{\raisebox{-0.2ex}{\begin{tikzpicture} \draw (0,0) circle (0.16); \draw[->] (0.16,0.03) -- (0.16,0.04);\end{tikzpicture}}\,}
\newcommand{\xl}{\textrm{\footnotesize\textcircled{\raisebox{-0.9ex}{\normalsize*}}}}
\newcommand{\xlw}{\textrm{\textcircled{\scriptsize\raisebox{0.2ex}{\textcircled{\raisebox{-1.3ex}{\normalsize*}}}}}}
\newcommand{\Zp}{Z^\bullet}
\newcommand{\Zc}{Z^\circ}
\newcommand{\tZ}{\tilde{Z}}
\newcommand{\expd}{\textrm{\textnormal{exp}}_\sqcup}
\newcommand{\Lk}{\textrm{\textnormal{Lk}}}
\newcommand{\lk}{\textrm{\textnormal{lk}}}
\newcommand{\T}{\boldsymbol{\top}}
\newcommand{\bK}{\hat{K}}
\newcommand{\bXi}{\hat{\Xi}}
\newcommand{\chir}{\mathbf{C}}
\newcommand{\chirs}{\mathbf{c}}
\newcommand{\totC}{\mathcal{C}}

\newcommand{\draww}[1]{\draw[white,line width=5pt] #1 \draw #1}

\newcommand{\rat}[3]{
\raisebox{-0.5cm}{
\begin{tikzpicture} [scale=0.7]
 \draw[dashed] (-1,1) node[above] {#3} -- (0,0) -- (0,1.1) (0,1) node[above] {#2} (0,0) -- (1,1) node[above] {#1}; 
\end{tikzpicture}}}

\title{Splitting formulas for the rational lift of the Kontsevich integral}
\author{Delphine Moussard}
\date{}

\begin{document}

\maketitle

\begin{abstract}
Kricker defined an invariant of knots in homology 3-spheres which is a rational lift of the Kontsevich integral 
and proved with Garoufalidis that this invariant satis\-fies splitting formulas with respect to a surgery move called null-move. 
We define a functorial extension of the Kricker invariant and prove splitting formulas for this functorial invariant 
with respect to null Lagrangian-preserving surgery, a genera\-li\-za\-tion of the null-move. We apply these splitting formulas 
to the Kricker invariant.
\vspace{1ex}


\end{abstract}

\tableofcontents

    \section{Introduction}

    \subsection{Context}

This paper presents the construction of a functorial extension of the Kricker rational lift of the Kontsevich integral, which aims at expliciting 
the properties of this invariant as a series of finite type invariants. 

The notion of finite type invariants first appeared in independent works of Goussarov and Vassiliev involving invariants of knots in the 3--dimensional 
sphere $S^3$; in this case, finite type invariants are also called Vassiliev invariants. 
Finite type invariants of knots in $S^3$ are defined by their polynomial behaviour with respect to crossing changes. 
The discovery of the Kontsevich integral, which is a universal invariant among all finite type invariants of knots in $S^3$, revealed that this class 
of invariants is very prolific. It is known, for instance, that it dominates all Witten-Reshetikhin-Turaev's quantum invariants. 
A theory of finite type invariants can be defined for any kind of topological objects provided that an elementary move on the set of these objects is fixed; 
the finite type invariants are defined by their polynomial behaviour with respect to this elementary move. 
For 3--dimensional manifolds, the notion of finite type invariants was introduced by Ohtsuki \cite{Oht4}, who constructed the first examples for integral 
homology 3--spheres, and it has been widely developed and generalized since then. In particular, Goussarov and Habiro independently developed 
a theory which involves any 3--dimensional manifolds ---and their knots--- and which contains the Ohtsuki theory for $\Z$--spheres \cite{GGP,Hab}. 

In \cite{Kri}, Kricker constructed a rational lift of the Kontsevich integral of knots in integral homology 3--spheres ($\Z$--spheres). In \cite{GK}, he proved with 
Garoufalidis that his construction provides an invariant of knots in $\Z$--spheres. They also proved that the Kricker invariant satisfies some splitting 
formulas with respect to the so-called null-move. For knots in $\Z$--spheres with trivial Alexander polynomial, these formulas together with work of Garoufalidis 
and Rozansky \cite{GR} imply that the Kricker invariant is a universal finite type invariant with respect to the null-move. 

Kricker's construction easily 
generalizes to null-homologous knots in rational homology 3--spheres ($\Q$--spheres); the main goal of this article is to prove splitting formulas for the Kricker 
invariant of these knots with respect to null Lagrangian-preserving surgery, a move which generalizes the null-move. 
For null-homologous knots in $\Q$--spheres with trivial Alexander polynomial, these formulas are used in \cite{M7} to prove that this extended Kricker 
invariant is a universal finite type invariant with respect to null Lagrangian-preserving surgeries. 

Lescop defined in \cite{Les2} an invariant of null-homologous knots in $\Q$--spheres and proved in \cite{Les3} splitting formulas for this invariant 
with respect to null Lagrangian-preserving surgeries, similar to the ones proved in this paper for the Kricker invariant. Lescop conjectured in \cite{Les3} that her invariant 
is equivalent to the Kricker invariant. The mentioned results of Garoufalidis, Kricker, Lescop and Rozansky give such an equivalence for knots in $\Z$--spheres 
with trivial Alexander polynomial and the results of the present paper allow to generalize this equivalence to null-homologous knots in $\Q$--spheres with 
trivial Alexander polynomial \cite[Theorem 2.11]{M7}. 

A similar situation arises in the study of finite type invariants of $\Q$--spheres with respect to Lagrangian-preserving surgeries. In this case, 
the Kontsevich--Kuperberg--Thurston invariant and the Le--Murakami--Ohtsuki invariant are both universal, up to degree 1 invariants deduced from the cardinality 
of the first homology group; this implies an equivalence result for these two invariants, see \cite{M2}. For the KKT invariant, splitting formulas with respect 
to Lagrangian-preserving surgeries were proved by Lescop \cite{Les}; for the LMO invariant, similar formulas were proved by Massuyeau \cite{Mas}. 
Massuyeau's proof of his splitting formulas is based on an extension of the LMO invariant of $\Q$--spheres to a functor defined on a category of Lagrangian 
cobordisms that he constructed with Cheptea and Habiro \cite{CHM}. 

In this paper, we extend the LMO functorial invariant of Cheptea--Habiro--Massuyeau to a category of Lagrangian cobordisms with paths, inserting the Kricker's 
idea in the construction. We obtain a functorial invariant from which the Kricker invariant of null-homologous knots in $\Q$--spheres 
is recovered. Following Massuyeau, we use the functo\-ria\-lity to obtain splitting formulas for our invariant and, as a consequence, 
for the Kricker invariant. 

\paragraph{Notations and conventions.}
For $\K=\Z,\Q$, a {\em $\K$--sphere}, (resp. a {\em $\K$--cube}) is a 3--manifold, compact and oriented, which has the same homology with coefficients 
in $\K$ as the standard 3--sphere (resp. 3--cube). The boundary of an oriented manifold with boundary is oriented with the ``outward normal first'' convention. 

\paragraph{Acknowledgments.}
I am supported by a Postdoctoral Fellowship of the Ja\-pan Society for the Promotion of Science. I am grateful to Tomotada Ohtsuki and the Research Institute for Mathematical Sciences for their support. I also wish to thank Gw\'ena\"el Massuyeau for interesting exchanges. Finally, I thank the referee whose useful comments helped to improve the paper.

    \subsection{Statement of the main result}

We first give the definitions we need to state our main result. 

\paragraph{Null LP--surgeries.}
For $g\in \mathbb{N}$, a \emph{genus $g$ rational homology handlebody ($\Q$--handlebody)} is a 3--manifold which is compact, oriented, 
and which has the same homology with rational coefficients as the standard genus $g$ handlebody. 
Such a $\Q$--handlebody is connected, and its boundary is necessarily homeomorphic to the standard genus $g$ surface.

The \emph{Lagrangian} $\mathcal{L}_C$ of a $\Q$--handlebody $C$ is the kernel of the map $i_*: H_1(\partial C;\Q)\to H_1(C;\Q)$
induced by the inclusion. The Lagrangian of a $\Q$--handlebody $C$ is indeed a Lagrangian subspace of $H_1(\partial C;\Q)$ 
with respect to the intersection form. A {\em Lagrangian-preserving pair}, or {\em LP--pair}, is a pair $\chir=\left(\frac{C'}{C}\right)$ 
of $\Q$--handlebodies equipped with a homeomorphism $h:\partial C\fl{\cong}\partial C'$ such that $h_*(\mathcal{L}_C)=\mathcal{L}_{C'}$. 

Given a 3--manifold $M$, a \emph{Lagrangian-preserving surgery}, or \emph{LP--surgery}, on $M$ is a family $\chir=(\chir_1,\dots,\chir_n)$ 
of LP--pairs such that the $C_i$ are embedded in $M$ and disjoint. The manifold obtained from $M$ by LP--surgery on $\chir$ is defined as 
$$M(\chir)=\left(M\setminus(\sqcup_{1\leq i \leq n} C_i)\right)\cup_{\partial}(\sqcup_{1\leq i \leq n} C_i').$$

Let $M$ be a 3--manifold such that $H_1(M;\Q)=0$ and let $K$ be a disjoint union of knots or paths properly embedded in $M$. A \emph{$\Q$--handlebody null in $M\setminus K$} is a $\Q$--handlebody $C\subset M\setminus K$ such that the map $i_* : H_1(C;\Q)\to H_1(M\setminus K;\Q)$ induced by the inclusion has a trivial image.
A \emph{null LP--surgery} on $(M,K)$ is an LP--surgery $\chir=(\chir_1,\dots,\chir_n)$ on $M\setminus K$ such that each $C_i$ is null in $M\setminus K$. 
The pair obtained by surgery is denoted by $(M,K)(\chir)$. 

\paragraph{The tensor $\mu(\chir)$.}
Given an LP--pair $\chir=\left(\frac{C'}{C}\right)$, define the associated {\em total manifold} $\totC=(-C)\cup C'$ and define 
$$\mu(\chir)\in\hom(\Lambda^3 H^1(\totC;\Q),\Q)\cong\Lambda^3 H_1(\totC;\Q)$$
by associating with a triple of cohomology classes the evaluation of their triple cup product on the fundamental class of $\totC$. 
For a family $\chir=(\chir_1,\dots,\chir_n)$ of LP--pairs, let $\totC=\totC_1\sqcup\dots\sqcup\totC_n$ and set:
$$\mu(\chir)=\mu(\chir_1)\otimes\dots\otimes\mu(\chir_n)\ \in\otimes_{i=1}^n\Lambda^3 H_1(\totC_i;\Q).$$
The natural identification $H_1(\totC;\Q)\cong\oplus_{i=1}^n H_1(\totC_i;\Q)$ allows to see $\otimes_{i=1}^n\Lambda^3 H_1(\totC_i;\Q)$ as a subspace of $\left(\Lambda^3 H_1(\totC;\Q)\right)^{\otimes n}$. This subspace injects to $S^n\Lambda^3 H_1(\totC;\Q)$ {\em via} the canonical surjection $\left(\Lambda^3 H_1(\totC;\Q)\right)^{\otimes n}\twoheadrightarrow S^n\Lambda^3 H_1(\totC;\Q)$. Hence we can view $\mu(\chir)$ as an element of $S^n\Lambda^3 H_1(\totC;\Q)$.

\paragraph{The bilinear form $\ell_{(S,\kappa)}(\chir)$.}
Let $(S,\kappa)$ be a {\em $\Q$SK--pair}, {\em i.e.} a pair made of a $\Q$--sphere $S$ and a null-homologous knot $\kappa\subset S$. 
Let $\chir=(\chir_1,\dots,\chir_n)$ be a null LP--surgery on $(S,\kappa)$. Let $\totC=\totC_1\sqcup\dots\sqcup\totC_n$ 
be the disjoint union of the associated total manifolds. Fix a lift $\tilde{C}_i$ of each $C_i$ in $\tilde{E}$. We will define a {\em hermitian} form:
$$\ell_{(S,\kappa)}(\chir): H_1(\totC;\Q)\times H_1(\totC;\Q)\to\Qt,$$
{\em i.e.} a $\Q$--bilinear form such that reversing the order of the arguments changes $t$ to $t^{-1}$. Let $a\in H_1(\totC_i;\Q)$ and $b\in H_1(\totC_j;\Q)$ 
be homology classes that can be represented by simple closed curves $\alpha\subset\partial C_i$ and $\beta\subset\partial C_j$, disjoint if $i=j$. 
Note that such homology classes generate $H_1(\totC;\Q)$ over $\Q$. 
Let $\tilde{\alpha}$ and $\tilde{\beta}$ be the copies of $\alpha$ and $\beta$ in $\tilde{C}_i$ and $\tilde{C}_j$. Set:
$$\ell_{(S,\kappa)}(\chir)(a,b)=\lk_e(\tilde{\alpha},\tilde{\beta}),$$
where $\lk_e(\cdot,\cdot)$ stands for the equivariant linking number (see for instance \cite[Section~2.1]{M7} for a definition). 
We get a well-defined hermitian form $\ell_{(S,\kappa)}(\chir)$ associated with a choice of lifts of the $C_i$'s. We will keep this choice implicit; 
the statement of Theorem~\ref{thmain} is valid for any such choice. 

\paragraph{Diagrammatic representations.}
Let $V$ be a rational vector space. A {\em $V$--colored Jacobi diagram} is a unitrivalent graph whose trivalent vertices are oriented and whose univalent 
vertices are labelled by $V$, where an {\em orientation} of a trivalent vertex is a cyclic order of the three edges that meet at this vertex ---fixed 
as \raisebox{-1.5ex}{
\begin{tikzpicture} [scale=0.2]
\newcommand{\tiers}[1]{
\draw[rotate=#1,color=white,line width=4pt] (0,0) -- (0,-2);
\draw[rotate=#1] (0,0) -- (0,-2);}
\draw (0,0) circle (1);
\draw[<-] (-0.05,1) -- (0.05,1);
\tiers{0}
\tiers{120}
\tiers{-120}
\end{tikzpicture}}
in the pictures. Set:
$$\A_\Q(V)=\frac{\Q\langle V\textrm{--colored Jacobi diagrams}\rangle}{\Q\langle\textrm{AS, IHX, LV}\rangle},$$
where the relations are depicted in Figure \ref{figrelations1}.
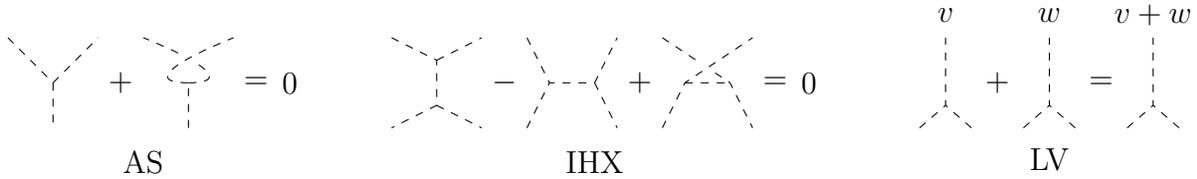
\begin{figure}[htb] 
\begin{center}
\begin{tikzpicture} [scale=0.3]
\begin{scope} [xshift=1cm]
 \draw[dashed] (0,4) -- (2,2);
 \draw[dashed] (2,2) -- (4,4);
 \draw[dashed] (2,2) -- (2,0);
 \draw (5,2) node{$+$};
 \draw[dashed] (8,2) .. controls +(2,0) and +(2.5,-1) .. (6,4);
 \draw[white,line width=6pt] (8,2) .. controls +(-2,0) and +(-2.5,-1) .. (10,4);
 \draw[dashed] (8,2) .. controls +(-2,0) and +(-2.5,-1) .. (10,4);
 \draw[dashed] (8,0) -- (8,2);
 \draw (11,2) node{$=$};
 \draw (12.5,2) node{0};
 \draw (6,-1.5) node{AS};
\end{scope}
 \draw[dashed] (18,4) -- (20,3) -- (20,1) -- (18,0);
 \draw[dashed] (20,1) -- (22,0);
 \draw[dashed] (20,3) -- (22,4);
 \draw (23,2) node{$-$};
 \draw[dashed] (24,4) -- (25,2) -- (27,2) -- (28,4);
 \draw[dashed] (24,0) -- (25,2);
 \draw[dashed] (27,2) -- (28,0);
 \draw (29,2) node{$+$};
 \draw[dashed] (30,4) -- (33,2) -- (34,0);
 \draw[white,line width=6pt] (31,2) -- (34,4);
 \draw[dashed] (30,0) -- (31,2) -- (34,4);
 \draw[dashed] (31,2) -- (33,2);
 \draw (35,2) node{$=$};
 \draw (36.5,2) node{0};
 \draw (27,-1.5) node{IHX};
\begin{scope} [xscale=2.3,yscale=2,xshift=18.5cm]
 \foreach \x in {0,2,4} {
 \draw[dashed] (\x-0.5,0) -- (\x,0.5) -- (\x,2) (\x+0.5,0) -- (\x,0.5);}
 \draw (0,2.5) node {$v$} (1,1) node {$+$} (2,2.5) node {$w$} (3,1) node {$=$} (4,2.5) node {$v+w$};
 \draw (2,-0.7) node{LV};
\end{scope}
\end{tikzpicture}
\end{center} \caption{Relations AS, IHX and LV on Jacobi diagrams.} \label{figrelations1}
\end{figure}
A symmetric tensor in $S^n\Lambda^3 V$ can be represented by a Jacobi diagram {\em via} the following embedding. 
$$\begin{array}{c c c}
 S^n\Lambda^3 V & \to & \A_\Q(V) \\
 (u_1\wedge v_1\wedge w_1)\dots(u_n\wedge v_n\wedge w_n) & \mapsto & \rat{$u_1$}{$v_1$}{$w_1$}\sqcup\dots\sqcup\rat{$u_n$}{$v_n$}{$w_n$}
\end{array}$$
Now define a {\em $\Qt$--beaded Jacobi diagram} as a trivalent graph whose vertices are oriented and whose edges are oriented and labelled by $\Qt$. 
Set:
$$\tA_{\Qt}(\varnothing)=\frac{\Q\langle \Qt\textrm{--beaded Jacobi diagrams}\rangle}{\Q\langle\textrm{AS, IHX, LE, Hol, OR}\rangle},$$
where the relations are depicted in Figures \ref{figrelations1} and \ref{figrelations2}, 
\begin{figure}[htb] 
\begin{center}
\begin{tikzpicture} [scale=0.3]
\begin{scope} 
 \draw (0,2) node{$x$};
 \draw[dashed] (1,0) -- (1,4);
 \draw[->,dashed] (1,0) -- (1,3);
 \draw (1.8,2.8) node{$P$};
 \draw (3.2,2) node{$+$};
 \draw (4.7,2) node{$y$};
 \draw[dashed] (5.7,0) -- (5.7,4);
 \draw[->,dashed] (5.7,0) -- (5.7,3);
 \draw (6.5,2.8) node{$Q$};
 \draw (7.9,2) node{$=$};
 \draw[dashed] (9.4,0) -- (9.4,4);
 \draw[->,dashed] (9.4,0) -- (9.4,3);
 \draw (12.4,2.8) node{$xP+yQ$};
 \draw (6.5,-1.5) node{LE};
\end{scope}
\begin{scope} [xshift=22cm,yshift=1cm,scale=0.9]
 \newcommand{\edge}[1]{
 \draw[rotate=#1,dashed] (0,0) -- (0,3);
 \draw[rotate=#1,->,dashed] (0,0) -- (0,1.5);}
 \edge{0} \draw (0,1.5) node[right] {$P$};
 \edge{120} \draw (-1.5,-0.75) node[above] {$Q$};
 \edge{240} \draw (1,-0.9) node[above right] {$R$};
 \draw (4,0) node{$=$};
\begin{scope} [xshift=8cm]
 \edge{0} \draw (0,1.5) node[right] {$tP$};
 \edge{120} \draw (-1.5,-0.75) node[above] {$tQ$};
 \edge{240} \draw (1,-0.9) node[above right] {$tR$};
\end{scope}
 \draw (4,-3.5) node{Hol};
\end{scope}
\begin{scope} [xshift=38cm]
 \draw[dashed] (0,0) -- (0,4);
 \draw[->] (0,2.9) -- (0,3) node[right] {$P(t)$};
 \draw (4,2) node{$=$};
 \draw[dashed] (5.5,0) -- (5.5,4);
 \draw[->] (5.5,3.1) -- (5.5,3) node[right] {$P(t^{-1})$};
 \draw (4,-1.5) node{OR};
\end{scope}
\end{tikzpicture}
\end{center} \caption{Relations LE, Hol and OR on Jacobi diagrams.} \label{figrelations2}
\end{figure}
with the IHX relation defined with the central edge labelled by 1. 
Define the {\em i--degree}, or {\em internal degree}, of any Jacobi diagram as its number of trivalent vertices. Given a hermitian form $\ell:V\times V\to\Qt$, 
one can {\em glue with $\ell$} some legs of a $V$--colored Jacobi diagram as depicted in Figure \ref{figglue}. 
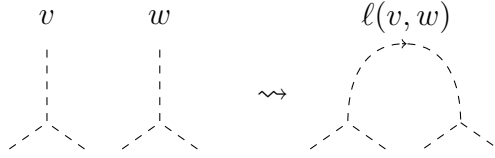
\begin{figure}[htb]
\begin{center}
\begin{tikzpicture} [yscale=0.7]
 \foreach \x in {0,1.5} {
 \draw[dashed] (\x-0.5,0) -- (\x,0.5) -- (\x,2) (\x+0.5,0) -- (\x,0.5);}
 \draw (0,2.5) node {$v$} (1.5,2.5) node {$w$} (3,1) node {$\rightsquigarrow$};
 \draw[dashed] (4.5,0) -- (4,0.5) (3.5,0) -- (4,0.5) .. controls +(0,0.9) and +(-0.5,0) .. (4.75,2) .. controls +(0.5,0) and +(0,0.9) .. (5.5,0.5) -- (5,0) (5.5,0.5) -- (6,0);
 \draw[->] (4.72,2) -- (4.78,2) node[above] {$\ell(v,w)$};
\end{tikzpicture}
\end{center} \caption{Gluing some legs of a Jacobi diagram with $\ell$.} \label{figglue}
\end{figure}
If $n$ is even, one can pairwise glue all legs of a $V$--colored Jacobi diagram of i--degree $n$ in order to 
get an element of $\tA_{\Qt}(\varnothing)$. This latter space is the target space of the Kricker invariant of $\Q$SK--pairs. 

We can now state our main result, about the Kricker invariant $\tZ$, proved in Section~\ref{secformules}. Note that null LP--surgeries define a move on the set of $\Q$SK--pairs. 
\begin{theorem} \label{thmain}
 Let $(S,\kappa)$ be a $\Q$SK--pair. Let $\chir=(\chir_1,\dots,\chir_n)$ be a null LP--surgery on $(S,\kappa)$. Then:
 $$\sum_{I\subset\{1,\dots,n\}}(-1)^{|I|}\tZ\left((S,\kappa)(\chir_I)\right)\equiv_n
 \left(\textrm{\begin{partext}{6cm} sum of all ways of gluing all legs of $\mu(\chir)$ with $\ell_{(S,\kappa)}(\chir)/2$\end{partext}}\right),$$
 where $\equiv_n$ means ``equal up to terms of i--degree at least $n+1$''.
\end{theorem}

\paragraph{Example.}
Let $(S^3,\mathcal{O})$ be the $\Q$SK--pair defined by the trivial knot $\mathcal{O}$ in the standard $3$--sphere. Let $C_1$ and $C_2$ be regular neighborhoods of the graphs $\Gamma_1$ and $\Gamma_2$ drawn in Figure \ref{figex}. 
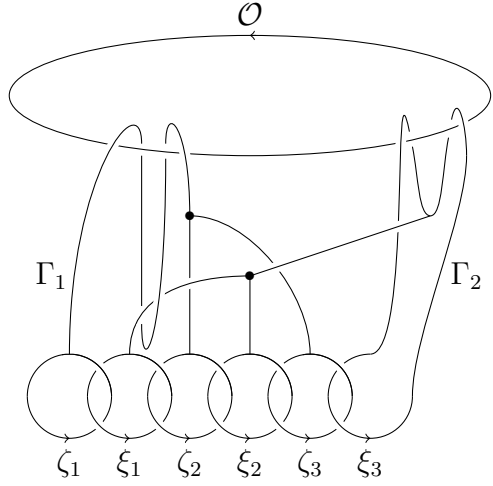
\begin{figure}[htb] 
\begin{center}
\begin{tikzpicture} [scale=0.8]
 \newcommand{\drawb}[1]{\draw[white,line width=4pt] #1; \draw #1;}
 \draw (4,5) ellipse (4 and 1);
 \draw (6,0.7) arc (90:180:0.7);
 \foreach \x in {5,...,1} {\drawb{(\x+0.7,0) arc (0:180:0.7)}}
 \foreach \x in {1,...,6} {\drawb{(\x-0.7,0) arc (-180:0:0.7)}}
 \foreach \x in {1,...,6} {\draw[->] (\x-0.05,-0.7) -- (\x,-0.7);}
 \foreach \x in {1,2,3} {\draw (2*\x-1,-0.7) node[below] {$\zeta_\x$};}
 \foreach \x in {1,2,3} {\draw (2*\x,-0.7) node[below] {$\xi_\x$};}
 \drawb{(2.6,4.2) .. controls +(0,0.7) and +(0,1.5) .. (3,3)}
 \draw (3,0.7) -- (3,3) node {$\scriptstyle{\bullet}$};
 \draw (3,3) .. controls +(1,0) and +(0,1) .. (5,0.7);
 \drawb{(4,2) .. controls +(-1,0) and +(0,1) .. (2,0.7)}
 \draw[->] (4.1,6) -- (4,6) node[above] {$\mathcal{O}$};
 \drawb{(1,0.7) .. controls +(0,3) and +(0,1) .. (2.2,4.2)}
 \draw (2.2,4) -- (2.2,1.5);
 \drawb{(2.2,1.2) .. controls +(0,-1) and +(0,-3) .. (2.6,3.95)}
 \foreach \x in {1,...,5} {\draw (\x+0.7,0) arc (0:100:0.7);}
 \drawb{(7.35,4.6) .. controls +(0.1,0.5) and +(0,0.5) .. (7.6,4)}
 \drawb{(6.65,4.4) .. controls +(-0.3,1.5) and +(0.7,0) .. (6,0.7)}
 \draw (7.6,4) .. controls +(0,-1) and +(0,1) .. (6.7,0);
 \drawb{(4,2) -- (7,3)}
 \draw (4,0.7) -- (4,2) node {$\scriptstyle{\bullet}$};
 \draw (7,3) .. controls +(-0.2,0) and +(0.1,-0.5) .. (6.7,4.15);
 \draw (7,3) .. controls +(0.2,0) and +(-0.1,-0.5) .. (7.3,4.3);
 \draw (0.7,2) node {$\Gamma_1$};
 \draw (7.6,2) node {$\Gamma_2$};
\end{tikzpicture} \caption{Surgery datum in $(S^3,\mathcal{O})$} \label{figex}
\end{center} 
\end{figure}
One can define an LP--surgery $\chir=(\chir_1,\chir_2)$ by associating with each $\Gamma_i$ a Borromean surgery, see for instance \cite[Section~2.2]{M7}. The associated tensor is given by $\mu(\chir_1)=\zeta_1\wedge\zeta_2\wedge\zeta_3$ and $\mu(\chir_2)=\xi_1\wedge\xi_2\wedge\xi_3$. There are fifteen ways to glue all legs of $$\mu(\chir)=\rat{$\zeta_1$}{$\zeta_2$}{$\zeta_3$}\rat{$\xi_1$}{$\xi_2$}{$\xi_3$}$$ with $\frac12\ell_{(S^3,\mathcal{O})}(\chir)$; all associated diagrams but one are trivial by the relation LE since they have a trivially labelled edge. Now $\lk_e(\zeta_1,\xi_1)=1$, $\lk_e(\zeta_2,\xi_2)=1$ and $\lk_e(\zeta_3,\xi_3)=t^{-1}$, so that we finally get:
$$\tZ\left((S^3,\mathcal{O})(\chir)\right)+\tZ\left((S^3,\mathcal{O})\right)\equiv_2 -\raisebox{-0.8cm}{
\begin{tikzpicture} [scale=0.3] 
  \draw[dashed] (0,0) .. controls +(0,2) and +(0,2) .. (4,0);
  \draw[dashed] (0,0) .. controls +(0,-2) and +(0,-2) .. (4,0);
  \draw[dashed] (0,0) node {$\scriptscriptstyle{\bullet}$} -- (4,0) node {$\scriptscriptstyle{\bullet}$};
  \draw[->] (1.95,-1.5) -- (2,-1.5) node[below] {$\scriptstyle{1}$};
  \draw[->] (1.95,0) -- (2,0) node[below] {$\scriptstyle{1}$};
  \draw[->] (1.95,1.5) -- (2,1.5) node[above] {$\scriptstyle{t}$};
 \end{tikzpicture}},$$
 where the vanishing of two terms in the left hand side is due to \cite[Lemma~2.2]{GGP}.

    \subsection{Strategy}

In this Subsection, we give a rough overview of the strategy developed to prove Theorem~\ref{thmain}. 

The main object of this article is the construction of a functorial LMO invariant defined on a category of Lagrangian cobordisms with paths. 
The morphisms of this category are cobordisms between compact surfaces with one boundary component, satisfying a Lagrangian-preserving condition, 
with finitely many disjoint paths with fixed extremities which we think of as knots with a fixed part on the boundary. 
This category is equivalent to a category of bottom-top tangles in $\Q$--cubes, whose top part has a trivial linking matrix, with paths with fixed extremities. 
These bottom-top tangles can be viewed as morphisms in a category of (general) tangles with paths in $\Q$--cubes, with an important difference in the composition law. 
Now a tangle with paths in a $\Q$--cube can be expressed as the result of a surgery on a link in a tangle with trivial paths ---segment lines--- in $[-1,1]^3$. 
To sum up, with a Lagrangian cobordism with paths, we associate a tangle with disks ---whose boundaries define the paths--- in $[-1,1]^3$ with a surgery link. 
This is represented in the first line of the scheme in Figure \ref{figscheme}. We initiate the construction of the invariant at the ``tangle with disks'' level. 
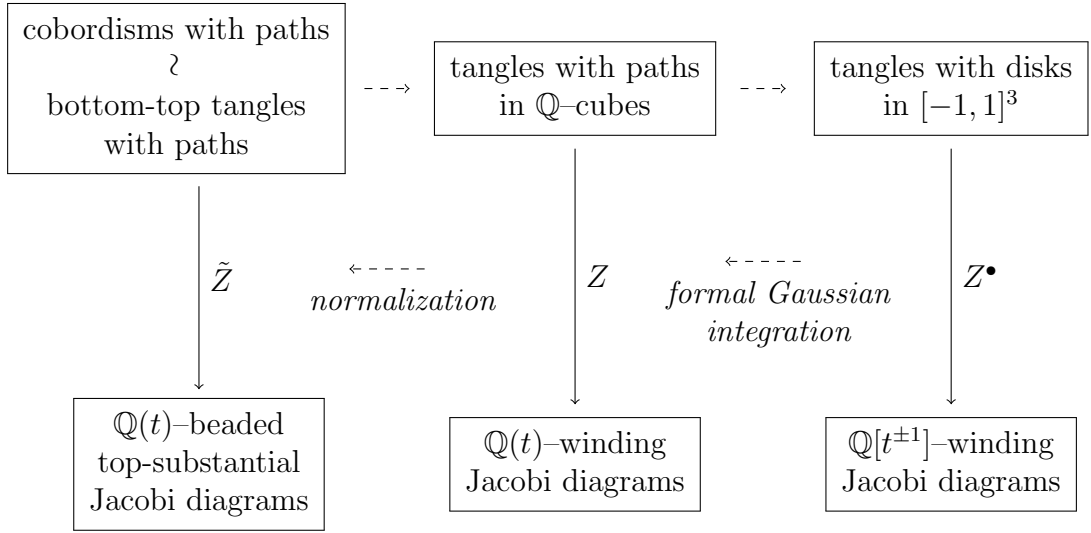
\begin{figure}[htb] 
\begin{center}
\begin{tikzpicture}
 \draw (-0.3,5) node {\framebox[\width]{\begin{tabular}{c}
                    cobordisms with paths \\
                    \rotatebox{90}{$\sim$} \\
                    bottom-top tangles \\ with paths
                   \end{tabular}}}; 
 \draw[dashed,->] (2.2,5) -- (2.8,5);
 \draw (5,5) node {\framebox[\width]{\begin{tabular}{c}
                                        tangles with paths \\ in $\Q$--cubes
                                       \end{tabular}}};
 \draw[dashed,->] (7.2,5) -- (7.8,5);
 \draw (10,5) node {\framebox[\width]{\begin{tabular}{c}
                     tangles with disks \\
                     in $[-1,1]^3$
                    \end{tabular}}};
 \draw[->] (0,3.7) -- (0,1);
 \draw[->] (5,4.2) -- (5,0.8);
 \draw[->] (10,4.2) -- (10,0.8);
 \draw (0,2.5) node[right] {$\tZ$} (5,2.5) node[right] {$Z$} (10,2.5) node[right] {$\Zp$};
 \draw (0,0) node {\framebox[\width]{\begin{tabular}{c}
                    $\Qt$--beaded \\
                    top-substantial \\
                    Jacobi diagrams
                   \end{tabular}}}; 
 \draw (5,0) node {\framebox[\width]{\begin{tabular}{c}
                                        $\Qt$--winding \\ Jacobi diagrams
                                       \end{tabular}}};
 \draw (10,0) node {\framebox[\width]{\begin{tabular}{c}
                     $\Qtt$--winding \\ Jacobi diagrams
                    \end{tabular}}};
 \draw[dashed,->] (3,2.6) -- (2,2.6);
 \draw (2.7,2.2) node {{\em normalization}};
 \draw[dashed,->] (8,2.7) -- (7,2.7);
 \draw (7.7,2) node {\begin{tabular}{c}
                        {\em formal Gaussian} \\ {\em integration}
                       \end{tabular}};
\end{tikzpicture}
\end{center} \caption{Scheme of construction for the invariant $\tZ$.} \label{figscheme}
\end{figure}

On the above mentioned categories, we define functorial invariants valued in categories of Jacobi diagrams with beads, {\em i.e.} unitrivalent graphs whose 
univalent vertices are labelled by some finite set or embedded in some 1--manifold ---the skeleton--- and whose edges are labelled (beaded) by powers of $t$, 
polynomials in $\Qtt$ or rational functions in $\Qt$. At the first step, we define a functor $\Zp$ on the category of tangles with disks by applying 
the Kontsevich integral and adding a bead $t^{\pm1}$ on the skeleton when the corresponding component meets a disk of the tangle. At a second step, 
we apply the invariant $\Zp$ to surgery presentations of tangles with paths in $\Q$--cubes. We use the formal Gaussian integration methods introduced by 
Bar-Natan, Garoufalidis, Rozansky and Thurston in \cite{AA1,AA2} and adapted to the beaded setting in \cite{Kri,GK}. We get a functor $Z$ on the category 
of tangles with paths in $\Q$--cubes. At the last step, given a Lagrangian cobordism with paths, we apply $Z$ to the associated bottom-top tangles with paths 
and normalize it following \cite{CHM} to obtain a functor $\tZ$ on the category of Lagrangian cobordisms with paths. Functoriality allows to prove 
splitting formulas for this invariant with respect to null Lagrangian-preserving surgeries.

Given a Lagrangian cobordism with one path between genus 0 surfaces, {\em i.e.} a $\Q$--cube with one path, one can glue a 3--ball to the boundary to get 
a $\Q$--sphere with a knot. In this way, the functor $\tZ$ provides an invariant of $\Q$SK--pairs which coincides with the Kricker 
invariant for knots in $\Z$--spheres. Splitting formulas for this invariant are deduced from the splitting formulas for the functor $\tZ$. 

\paragraph{Plan of the paper.}
We define the domain categories of cobordisms and tangles in Section \ref{secdomaincat}. In Section \ref{sectargetcat}, we define the target categories of 
Jacobi diagrams and gives the tools of formal Gaussian integration. Section \ref{secwinding} is devoted to the introduction of winding matrices, that will 
play the role of the linking matrices in the Cheptea--Habiro--Massuyeau construction of a functorial LMO invariant. The functors $\Zp$, $Z$ and $\tZ$ are constructed 
in Section \ref{secfunctor}. At the end of this section, from the functor $\tZ$, we deduce our version of the Kricker invariant for $\Q$SK--pairs; 
the behaviour of this invariant with respect to connected sum is stated. Finally, the splitting formulas are given in Section \ref{secformules}.

    \section{Domain categories: cobordisms and tangles} \label{secdomaincat}

  \subsection{Cobordisms with paths} \label{subseccob}

Given $g\in\N$, we fix a model surface $F_g$, compact, connected, oriented, of genus $g$, with one boundary 
component represented in Figure \ref{figFg}. 
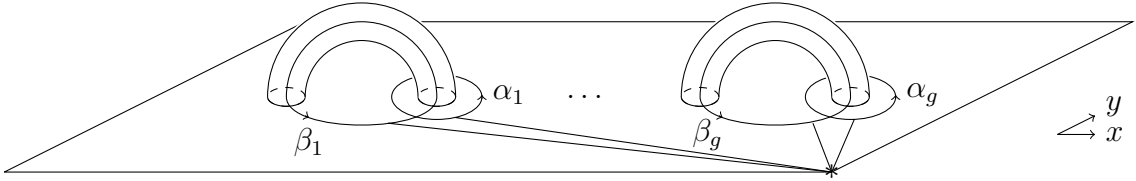
\begin{figure}[htb] 
\begin{center}
\begin{tikzpicture} [scale=0.5]
 \draw (0,0) -- (8,4) -- (30,4) -- (22,0) -- (0,0);
 \draw (22,0) node {$*$};
 \draw (11.5,2) circle (1.2 and 0.6) (12,1.45) -- (22,0);
 \draw[->] (12.7,1.95) -- (12.7,2.05) node[right] {$\alpha_1$};
 \draw (11.5,2) circle (0.5 and 0.25) (7.5,2) circle (0.5 and 0.25);
 \draw [white,line width=16pt] (11.5,2) arc (0:180:2);
 \draw (12,2) arc (0:180:2.5);
 \draw (11,2) arc (0:180:1.5);
 \draw[dashed] (11.5,2) circle (0.5 and 0.25) (7.5,2) circle (0.5 and 0.25);
 \draw (11.5,2) arc (0:180:2) .. controls +(0,-1) and +(0,-1) .. (11.5,2) (10.2,1.3) -- (22,0);
 \draw[->] (8,1.5) -- (8.1,1.45) node[below] {$\beta_1$};
 \draw (15.5,2) node {$\dots$};
 \draw (22.5,2) circle (1.2 and 0.6) (22.6,1.4) -- (22,0);
 \draw[->] (23.7,1.95) -- (23.7,2.05) node[right] {$\alpha_g$};
 \draw (22.5,2) circle (0.5 and 0.25) (18.5,2) circle (0.5 and 0.25);
 \draw [white,line width=16pt] (22.5,2) arc (0:180:2);
 \draw (23,2) arc (0:180:2.5);
 \draw (22,2) arc (0:180:1.5);
 \draw[dashed] (22.5,2) circle (0.5 and 0.25) (18.5,2) circle (0.5 and 0.25);
 \draw (22.5,2) arc (0:180:2) .. controls +(0,-1) and +(0,-1) .. (22.5,2) (21.5,1.3) -- (22,0);
 \draw[->] (19,1.5) -- (19.1,1.45);
 \draw (18.7,1) node {$\beta_g$};
 \draw[->] (28,1) -- (29,1) node[right] {$x$};
 \draw[->] (28,1) -- (29,1.5); 
 \draw (29,1.75) node[right] {$y$};
\end{tikzpicture}
\end{center} \caption{The model surface $F_g$.} \label{figFg}
\end{figure}
It is equipped with a fixed base point $*$ and a fixed basis $(\alpha_1,\dots,\alpha_g,\beta_1,\dots,\beta_g)$ of $\pi_1(F_g,*)$. 
Denote by $\cub$ the cube $[-1,1]^3$ with $g^+$ handles on the top 
boundary and $g^-$ tunnels in the bottom boundary. We have canonical embeddings $F_{g^+}\hookrightarrow\partial\cub$ and $F_{g^-}\hookrightarrow\partial\cub$. For $k\geq 0$ and $1\leq i\leq k$, set $h_i(k)=\frac{2i}{k+1}-1$. A {\em cobordism with paths} from $F_{g^+}$ to $F_{g^-}$ is an equivalence class of triples $(M,K,m)$ where:
\begin{itemize}
 \item $M$ is a compact, connected, oriented 3--manifold,
 \item $m:\partial\cub\fl{\cong}\partial M$ is an orientation-preserving homeomorphism,
 \item $K=\sqcup_{i=1}^k K_i\subset M$ is a union of $k\geq 0$ oriented paths $K_i$ from $m(0,1,h_i(k))$ to $m(0,-1,h_i(k))$,
 \item $\bK=\sqcup_{i=1}^k \bK_i$ is an oriented {\em boundary link}, {\em i.e.} the $\bK_i$ bound disjoint compact oriented surfaces in $M$, where $\bK_i$ is the knot defined as the union of $K_i$ with the line segments $[(0,-1,h_i(k)),(1,-1,h_i(k))]$, $[(1,-1,h_i(k)),(1,1,h_i(k))]$ and 
  $[(1,1,h_i(k)),(0,1,h_i(k))]$.
\end{itemize}
\begin{figure}[htb]
\begin{center}
\begin{tikzpicture} [scale=0.2]
 \newcommand{\surface}[2]{
 \begin{scope} [xshift=#1cm,yshift=#2cm]
 \draw (0,0) -- (8,4) -- (30,4) -- (22,0) -- (0,0);
 \draw (11.5,2) circle (0.5 and 0.25) (7.5,2) circle (0.5 and 0.25);
 \draw [white,line width=8pt] (11.5,2) arc (0:180:2);
 \draw[densely dashed] (11.5,2) circle (0.5 and 0.25) (7.5,2) circle (0.5 and 0.25);
 \draw (12,2) arc (0:180:2.5);
 \draw (11,2) arc (0:180:1.5);
 \draw (15.5,2) node {$\dots$};
 \draw (22.5,2) circle (0.5 and 0.25) (18.5,2) circle (0.5 and 0.25);
 \draw [white,line width=8pt] (22.5,2) arc (0:180:2);
 \draw[densely dashed] (22.5,2) circle (0.5 and 0.25) (18.5,2) circle (0.5 and 0.25);
 \draw (23,2) arc (0:180:2.5);
 \draw (22,2) arc (0:180:1.5);
 \end{scope}}
 \foreach \x in {30,52} {\draw (\x,4) -- (\x,19);}
 \surface{22}{0}
 \foreach \y in {5,10} {\draw[red,dashed] (52,\y+4) -- (41,\y+4);}
 \draw[red] (33,5) .. controls +(2,1) and +(2,-1) .. (40,7) .. controls +(-2,1) and +(-1,2) .. (41,9);
 \draw[red,->] (33.9,12) -- (33.8,11.9) node[left] {$K_2$};
 \draw[red] (33,10) .. controls +(0,2) and +(0,2) .. (36,12) .. controls +(0,-2) and +(0,-2) .. (38,14) .. controls +(0,2) and +(0,2) .. (41,14);
 \draw[red,->] (36,5.72) -- (35.9,5.7) node[above] {$K_1$};
 \draw[white,line width=3pt] (22,15) -- (45,15);
 \surface{22}{15}
 \foreach \x in {22,44} {\draw[white,line width=3pt] (\x,0.5) -- (\x,14.6); \draw (\x,0) -- (\x,15);}
 \foreach \y in {5,10} {\draw[red,dashed] (33,\y) -- (44,\y) -- (52,\y+4);}
\end{tikzpicture} \caption{A Lagrangian cobordism with paths} \label{figLagcob}
\end{center}
\end{figure}
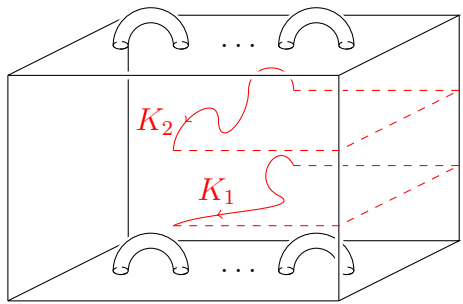
Two such triples are {\em equivalent} if they are related by an orientation-preserving homeomorphism which respects the boundary parametrizations 
and identifies the paths. We get embeddings $m_+:F_{g^+}\hookrightarrow\partial M$ and $m_-:F_{g^-}\hookrightarrow\partial M$.

Define a category $\tcob$ of cobordisms with paths whose objects are non-negative integers and whose set of morphisms $\tcob(g^+,g^-)$ is the set of 
cobordisms with paths from $F_{g^+}$ to $F_{g^-}$. The composition of a cobordism $(M,K,m)$ from $F_g$ to $F_f$ with a cobordism $(N,J,n)$ from $F_h$ to $F_g$ 
is given by gluing $N$ on the top of $M$, identifying $m_+(M)$ with $n_-(N)$ and reparametrizing the new manifold. The identity of $g\in\N$ is the cobordism 
$F_g\times[-1,1]$ with natural boundary parametrization and no path. 

Forgetting the datum of the paths in the cobordisms, one gets the category $\cob$ des\-cri\-bed in \cite{CHM}, which we view as the subcategory of $\tcob$ 
of cobordisms with no path. For a cobordism $(M,m)$ and a cobordism with paths $(N,J,n)$, define the tensor product $(M,m)\otimes(N,J,n)$ by horizontal 
juxtaposition in the $x$ direction. 

We now define the subcategory of $\tcob$ of Lagrangian cobordisms with paths. Set $A_g=\ker(\mathrm{incl}_*:H_1(F_g;\Q)\to H_1(C^g_0;\Q))$ 
and $B_g=\ker(\mathrm{incl}_*:H_1(F_g;\Q)\to H_1(C^0_g;\Q))$. These are Lagrangian subspaces of $H_1(F_g;\Q)$ with respect to the intersection form 
and $A_g$ (resp. $B_g$) is generated by the homology classes of the curves $\alpha_i$ (resp. $\beta_i$). 
A cobordism with paths $(M,K,m)$ from $F_{g^+}$ to $F_{g^-}$ is {\em Lagrangian(-preserving)} if the following conditions are satisfied:
\begin{itemize}
 \item $H_1(M;\Q)=(m_-)_*(A_{g^-})\oplus(m_+)_*(B_{g^+})$,
 \item $(m_+)_*(A_{g^+})\subset(m_-)_*(A_{g^-})$ as subspaces of $H_1(M;\Q)$. 
\end{itemize}
The Lagrangian property is preserved by composition, and we denote by $\tlcob$ the subcategory of $\tcob$ of Lagrangian cobordisms with paths. 
The subcategory of Lagrangian cobordisms with no path is the category $\lcob$ ---denoted by $\Q\lcob$ in \cite{Mas}. 

Define categories $\cob_q$, $\tcob_q$, $\lcob_q$ and $\tlcob_q$ of $q$--cobordisms with objects the non-commutative words in the single letter $\bullet$ and with set of morphisms from a word on $g^+$ letters to a word on $g^-$ letters the set of morphisms from $g^+$ to $g^-$ in $\cob$, $\tcob$, $\lcob$ and $\tlcob$ respectively. 

Let $(M,K,m)$ be a Lagrangian cobordism with paths from $g^+$ to $g^-$. Since the space $H_1(M;\Q)$ is non-trivial in general, we have to adapt the definition of null LP--surgeries given in the introduction. Let $N(K)$ be a tubular neighborhood of $K$ in $M$ and set $E=M\setminus\textrm{Int}(N(K))$. A standard homological computation gives $H_2(M;\Q)=0$, so that the exact sequence in homology associated with the pair $(M,E)$ provides the following short exact sequence:
$$0\to H_2(N(K),N(K)\cap E;\Q)\to H_1(E;\Q)\to H_1(M;\Q)\to 0.$$
The image of $H_2(N(K),N(K)\cap E;\Q)$ in $H_1(E;\Q)$ is a subspace $H_K\cong\Q^k$ generated by meridians of the components of $K$, where $k$ is the number of these components. Now, the parametrizations $m_{\pm}$ of the top and bottom boundaries of $M$ can be decomposed into injective maps as follows.
\begin{center}
\begin{tikzpicture} 
 \draw (0,-0.2) node {$E$};
 \draw (-1,1) node {$F_{g^\pm}$};
 \draw (1,1) node {$M$};
 \draw[->] (-0.5,1) -- (0.5,1);
 \draw (0,1) node[above] {$m_\pm$};
 \draw[->] (-0.7,0.7) -- (-0.2,0);
 \draw (-0.4,0.3) node[left] {$e_\pm$};
 \draw[<-] (0.7,0.7) -- (0.2,0);
\end{tikzpicture}
\end{center} 
Hence we have a canonical decomposition of $H_1(E;\Q)$ as 
$$H_1(E;\Q)=H_K\oplus(e_-)_*(A_{g^-})\oplus(e_+)_*(B_{g^+}),$$
where $(e_-)_*(A_{g^-})\oplus(e_+)_*(B_{g^+})\cong H_1(M;\Q)$ {\em via} the inclusion map. 
We say that a $\Q$--handlebody $C\subset M\setminus K$ is \emph{null with respect to $K$} if the composed map $$H_1(C;\Q)\fl{incl_*} H_1(M\setminus K;\Q)\longrightarrow H_K$$ has a trivial image, where the second map is the projection on the first factor in the above decomposition of $H_1(M\setminus K;\Q)\cong H_1(E;\Q)$.
A \emph{null LP--surgery} on $(M,K)$ is an LP--surgery $\chir=(\chir_1,\dots,\chir_n)$ on $M\setminus K$ such that each $C_i$ is null with respect to $K$.

   \subsection{Bottom-top tangles with paths} \label{subsecbtt}

Let us define the category of bottom-top tangles with paths. For a positive integer $g\geq0$, let $(p_1,q_1),\dots,(p_g,q_g)$ be $g$ pairs of points 
uniformly distributed on $[-1,1]\times\{0\}\subset[-1,1]^2\cong F_0$ as represented in Figure \ref{figpiqi}. 
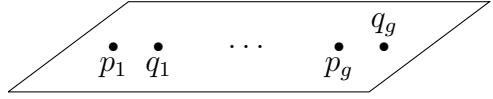
\begin{figure}[htb] 
\begin{center}
\begin{tikzpicture} [xscale=0.4,yscale=0.2]
 \draw (0,0) -- (4,6) -- (16,6) -- (12,0) -- (0,0);
 \draw (3.5,3) node {$\scriptstyle{\bullet}$};
 \draw (3.5,3) node[below] {$p_1$};
 \draw (5,3) node {$\scriptstyle{\bullet}$};
 \draw (5,3) node[below] {$q_1$};
 \draw (8,3) node {$\dots$};
 \draw (11,3) node {$\scriptstyle{\bullet}$};
 \draw (11,3) node[below] {$p_g$};
 \draw (12.5,3) node {$\scriptstyle{\bullet}$};
 \draw (12.5,3) node[above] {$q_g$};
\end{tikzpicture}
\end{center} \caption{The pairs of points $(p_i,q_i)$ on $[-1,1]^2$.} \label{figpiqi}
\end{figure}
A {\em bottom-top tangle with paths} of type $(g^+,g^-)$ 
is an equivalence class of triples $(B,K,\gamma)$ where: 
\begin{itemize}
 \item $(B,K)=(B,K,b)$ is a cobordism with paths form $F_0$ to $F_0$,
 \item $\gamma=(\gamma^+,\gamma^-)$ is a framed oriented tangle in $B$ with $g^+$ components $\gamma_i^+$ from $b(\{p_i\}\times\{1\})$ to $b(\{q_i\}\times\{1\})$ 
  and $g^-$ components $\gamma_i^-$ from $b(\{q_i\}\times\{-1\})$ to $b(\{p_i\}\times\{-1\})$,
 \item $\bK$ is a boundary link in $B\setminus\gamma$.
\end{itemize}
Two such triples $(B,K,\gamma)$ and $(B',K',\gamma')$ are {\em equivalent} if $(B,K)$ and $(B',K')$ 
are related by an equivalence which identifies $\gamma$ and $\gamma'$. 

In order to define the composition, we need the bottom-top tangle $([-1,1]^3,\varnothing,T_g)$ represented in Figure \ref{figTg}. 
\begin{figure}[htb] 
\begin{center}
\begin{tikzpicture} [xscale=0.4,yscale=0.3]
 \draw (0,0) -- (3,4) -- (20,4) (20,14) -- (3,14) (3,4) -- (3,14) (20,4) -- (20,14); 
 \draw (5.5,6.5) .. controls +(1,0) and +(0,-3) .. (7,12);
 \draw[white,line width=5pt] (4,2) .. controls +(0,3) and +(-1,0) .. (5.5,7.5) .. controls +(1,0) and +(0,3) .. (7,2);
 \draw (4,2) .. controls +(0,3) and +(-1,0) .. (5.5,7.5) .. controls +(1,0) and +(0,3) .. (7,2);
 \draw[white,line width=5pt] (4,12) .. controls +(0,-3) and +(-1,0) .. (5.5,6.5);
 \draw (4,12) .. controls +(0,-3) and +(-1,0) .. (5.5,6.5);
 \draw (14.5,6.5) .. controls +(1,0) and +(0,-3) .. (16,12);
 \draw[white,line width=5pt] (13,2) .. controls +(0,3) and +(-1,0) .. (14.5,7.5) .. controls +(1,0) and +(0,3) .. (16,2);
 \draw (13,2) .. controls +(0,3) and +(-1,0) .. (14.5,7.5) .. controls +(1,0) and +(0,3) .. (16,2);
 \draw[white,line width=5pt] (13,12) .. controls +(0,-3) and +(-1,0) .. (14.5,6.5);
 \draw (13,12) .. controls +(0,-3) and +(-1,0) .. (14.5,6.5);
 \draw[white, line width=5pt] (0,10) -- (17,10) -- (17,0);
 \draw (0,10) -- (17,10) -- (17,0) (20,4) -- (17,0) -- (0,0) -- (0,10) -- (3,14) (17,10) -- (20,14);
 \draw (10,2) node {$\dots$} (10,12) node {$\dots$};
 \draw (5.5,2) node[below] {$1$} (14.5,2) node[below] {$g$};
 \foreach \x in {4,7,13,16} \foreach \y in {2,12} \draw (\x,\y) node {$\scriptstyle{\bullet}$};
 \draw[->] (6.71,8.5) -- (6.74,8.6);
 \draw[->] (15.71,8.5) -- (15.74,8.6);
 \draw[->] (4.29,5.5) -- (4.26,5.4);
 \draw[->] (13.29,5.5) -- (13.26,5.4);
\end{tikzpicture}
\end{center} \caption{The bottom-top tangle $T_g$ in $[-1,1]^3$.} \label{figTg}
\end{figure}
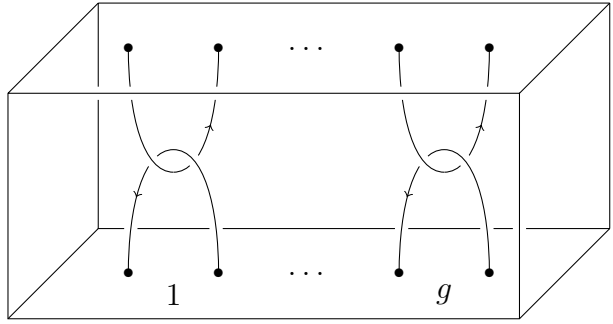
The composition of a bottom-top tangle 
$(B,K,\gamma)$ of type $(g,f)$ with a bottom-top tangle $(C,J,\upsilon)$ of type $(h,g)$ is given by first making the composition 
$(B,K)\circ([-1,1]^3,\varnothing)\circ(C,J)$ in the category $\tcob$ and then perfoming the surgery on the $2g$ components link $\gamma^+\cup T_g\cup\upsilon^-$. 
We get a category $\tbtt$ whose objects are non-negative integers and whose set of morphisms $\tbtt(g^+,g^-)$ is the set of bottom-top tangles with 
paths of type $(g^+,g^-)$. The identity of $g\in\N$ is the bottom-top tangle in $[-1,1]^3$ with no path represented in Figure \ref{figtangleIdg}.
\begin{figure}[htb] 
\begin{center}
\begin{tikzpicture} [xscale=0.4,yscale=0.3]
 \draw (0,0) -- (3,4) -- (20,4) (20,14) -- (3,14) (3,4) -- (3,14) (20,4) -- (20,14);
 \draw (4,12) .. controls +(0,-3) and +(-1,0) .. (5.5,6.5);
 \draww {(4,2) .. controls +(0,3) and +(-1,0) .. (5.5,7.5) .. controls +(1,0) and +(0,3) .. (7,2);}
 \draww {(5.5,6.5) .. controls +(1,0) and +(0,-3) .. (7,12);}
 \draw (13,12) .. controls +(0,-3) and +(-1,0) .. (14.5,6.5);
 \draww {(13,2) .. controls +(0,3) and +(-1,0) .. (14.5,7.5) .. controls +(1,0) and +(0,3) .. (16,2);}
 \draww {(14.5,6.5) .. controls +(1,0) and +(0,-3) .. (16,12);}
 \draw[white, line width=5pt] (0,10) -- (17,10) -- (17,0);
 \draw (0,10) -- (17,10) -- (17,0) (20,4) -- (17,0) -- (0,0) -- (0,10) -- (3,14) (17,10) -- (20,14);
 \draw (10,2) node {$\dots$} (10,12) node {$\dots$};
 \draw (5.5,2) node[below] {$1$} (14.5,2) node[below] {$g$};
 \foreach \x in {4,7,13,16} \foreach \y in {2,12} \draw (\x,\y) node {$\scriptstyle{\bullet}$};
 \draw[->] (6.71,8.5) -- (6.74,8.6);
 \draw[->] (15.71,8.5) -- (15.74,8.6);
 \draw[->] (4.29,5.5) -- (4.26,5.4);
 \draw[->] (13.29,5.5) -- (13.26,5.4);
\end{tikzpicture}
\end{center} \caption{The bottom-top tangle $Id_g$.} \label{figtangleIdg}
\end{figure}
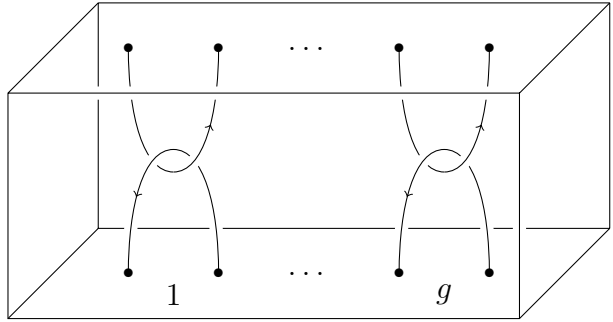

Forgetting the datum of the paths, one gets the category $\btt$ of bottom-top tangles introduced in \cite{CHM}, that we view as the subcategory of $\tbtt$ 
of bottom-top tangles with no path. For a bottom-top tangle $(B,\gamma)$ 
and a bottom-top tangle with paths $(C,J,\upsilon)$, define the tensor product $(B,\gamma)\otimes(C,J,\upsilon)$ by horizontal juxtaposition in the $x$ direction. 
Define categories $\btt_q$ and $\tbtt_q$ of bottom-top $q$--tangles with objects the non-commutative words in the single letter $\bullet$. 

The following result is a direct adaptation of \cite[Theorem 2.10]{CHM} which gives an isomorphism $D:\btt\to\cob$. The map $D$ is defined by digging 
tunnels around the components of the tangle. 
\begin{proposition} \label{propmapD}
 There is an isomorphism $D:\tbtt\to\tcob$ which identifies $\btt$ with $\cob$ and preserves the tensor product on $\btt\otimes\tbtt$. 
\end{proposition}
Let $(B,K,\gamma)$ be a bottom-top tangle with paths in a $\Q$--cube. Let $\bar{\gamma}$ be the link obtained by closing the components of $\gamma$ 
with the line segments $[(p_i,\pm1),(q_i,\pm1)]$. Define the linking matrix $\Lk(\gamma)$ of $\gamma$ in $B$ with the linkings of the components 
of $\bar{\gamma}$. The characte\-ri\-za\-tion of the bottom-top tangles sent onto Lagrangian cobordisms by $D$ 
given in \cite[Lemma 2.12]{CHM} directly generalizes to: 
\begin{lemma}
 Given a bottom-top tangle with paths $(B,K,\gamma)$, the cobordism with paths $D(B,K,\gamma)$ is Lagrangian if and only if $B$ is a $\Q$--cube 
 and $\Lk(\gamma^+)$ is trivial. 
\end{lemma}

    \subsection{Tangles with paths and tangles with disks} \label{subsectangles}

Given a cobordism $(B,b)$ from $F_0$ to $F_0$, a {\em tangle $\gamma$ in $B$} is an isotopy (rel. $\partial B$) class of framed oriented tangles whose boundary 
points lie on the top and bottom surfaces and are uniformly distributed along the line segments $[-1,1]\times\{0\}\times\{1\}$ and $[-1,1]\times\{0\}\times\{-1\}$ 
in $\partial B=b(\partial[-1,1]^3)$. 
Associate with each boundary point of $\gamma$ the sign $+$ if $\gamma$ is oriented downwards at that point and the sign $-$ otherwise. This provides two words 
in the letters $+$ and $-$, one for the top surface and the other for the bottom surface. Lifting these two words into non-associative words $w_t(\gamma)$ 
and $w_b(\gamma)$ in the letters $(+,-)$, one gets a {\em $q$--tangle}. A $q$--tangle $\gamma$ in a cobordism with paths $(B,K,b)$ defines 
a {\em $q$--tangle with paths} $(B,K,\gamma)$ if $\bK$ is a boundary link in $B\setminus\gamma$. 

Define two categories $\tcub$ and $\ttcub$ with objects the non-associative words in the letters $(+,-)$ and morphisms the $q$--tangles in $\Q$--cubes for $\tcub$ 
and the $q$--tangles with paths in $\Q$--cubes for $\ttcub$, up to orientation-preserving homeomorphism respecting the boundary parametrization. 
Composition is given by vertical juxtaposition. Given a morphism $(C,\upsilon)$ in $\tcub$ and a morphism $(B,K,\gamma)$ in $\ttcub$, 
define the tensor product $(C,\upsilon)\otimes(B,K,\gamma)$ by horizontal juxtaposition in the $x$ direction. 

\begin{lemma} \label{lemmapreschir}
 Let $(B,K,\gamma)$ be a $q$--tangle with paths in a $\Q$--cube. There exist a $q$--tangle with paths $([-1,1]^3,\Xi,\eta)$ and a framed link 
 $L\subset[-1,1]^3\setminus(\Xi\cup\eta)$, with $\Xi$ a union of line segments and $L$ null-homotopic in $[-1,1]^3\setminus\Xi$, 
 such that $(B,K,\gamma)$ is obtained from $([-1,1]^3,\Xi,\eta)$ by surgery on $L$. Moreover, two such surgery links 
 are related by the follo\-wing Kirby moves: the blow-up/blow-down move KI which adds or removes a split trivial component with framing $\pm1$ unknotted with $\eta$, 
 and the handleslide move KII which adds a surgery component to another surgery component or to a component of the tangle $\eta$ (see Figure~\ref{figK2}).
 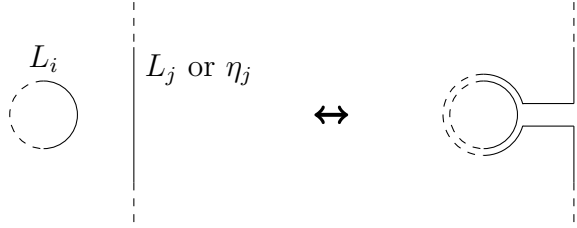
\begin{figure}[htb] 
 \begin{center}
 \begin{tikzpicture} [scale=0.3]
  \draw (0,0) -- (0,-6);
  \draw[dashed] (0,2) -- (0,0) (0,-6) -- (0,-8);
  \draw (-4,-4.5) arc (-90:90:1.5);
  \draw[dashed] (-4,-1.5) arc (90:270:1.5);
  \draw (-4,-0.5) node{$L_i$};
  \draw (0,-1) node[right] {$L_j$ or $\eta_j$};
  \draw[<->,line width=1.5pt] (8,-3) -- (9.5,-3);
 \begin{scope} [xshift=19.5cm]
  \draw (0,0) -- (0,-2.5) (0,-3.5) -- (0,-6);
  \draw[dashed] (0,2) -- (0,0) (0,-6) -- (0,-8);
  \draw (-4,-4.5) arc (-90:90:1.5);
  \draw[dashed] (-4,-1.5) arc (90:270:1.5);
  \draw (-4,-4.8) arc (-90:90:1.8);
  \draw[dashed] (-4,-1.2) arc (90:270:1.8);
  \draw[color=white,line width=3pt] (-2.26,-2.53) -- (-2.26,-3.48);
  \draw (-2.26,-2.5) -- (0,-2.5) (-2.26,-3.5) -- (0,-3.5);
 \end{scope}
 \end{tikzpicture} \caption{KII move.} \label{figK2}
 \end{center}
 \end{figure}
\end{lemma}
\begin{proof}
 Let $\Sigma$ be a Seifert surface of $\bK$ which is the disjoint union of Seifert surfaces of the $\bK_i$. Choose $\Sigma$ disjoint from $\gamma$. 
 Take a link $J\subset B$ such that surgery on $J$ gives $[-1,1]^3$. Performing isotopies on $J$ if necessary, we can assume that $J$ does not meet $\Sigma$. 
 The handles of (the image in $[-1,1]^3$ of) $\Sigma$ can be unlinked 
 by adding surgery components as shown in Figure \ref{figundocrossing}. 
 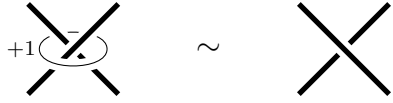
\begin{figure} [htb]
 \begin{center}
 \begin{tikzpicture} [scale=0.3]
 \newcommand{\drawwl}[1]{\draw[white,line width=6pt] #1 \draw[line width=2pt] #1}
 \begin{scope}
  \draw (0.5,2) .. controls +(0,1) and +(0,1) .. (3.5,2);
  \drawwl {(0,4) -- (4,0);}
  \drawwl {(0,0) -- (4,4);}
  \draww {(0.5,2) .. controls +(0,-1) and +(0,-1) .. (3.5,2);}
  \draw (-0.3,2) node {$\scriptstyle{+1}$};
 \end{scope}
  \draw (8,2) node {$\sim$};
 \begin{scope} [xshift=12cm]
  \draw[line width=2pt] (0,0) -- (4,4);
  \drawwl {(0,4) -- (4,0);}
 \end{scope}
 \end{tikzpicture}
 \end{center} \caption{Surgery changing a crossing.} \label{figundocrossing}
 \end{figure}
 In this way, $\bK$ can be turned into a trivial link. This provides a surgery link in $B\setminus(K\cup\gamma)$, disjoint 
 from $\Sigma$, such that surgery on this link changes $(B,K,\gamma)$ to a $q$--tangle with paths $([-1,1]^3,\Xi,\eta)$ 
 as required. Let $L$ be an inverse surgery link. It is null-homotopic in $[-1,1]^3\setminus\Xi$ since it is disjoint from $\Sigma$. 
 
 For the last assertion, apply \cite[Theorem 3.1]{HW} in $[-1,1]^3\setminus\Xi$. Note that a split trivial component with framing $\pm1$ can always be unknotted 
 from $\eta$ using the KII move.
\end{proof}
A family $(([-1,1]^3,\Xi,\eta),L)$ satisfying the conditions of the lemma with $L$ oriented is a {\em surgery presentation} of $(B,K,\gamma)$. 
When $\gamma$ (and thus $\eta$) is a bottom-top tangle, the components of $\eta$ can be closed by line segments in the top and bottom surfaces. The obtained 
curves are null-homotopic in $[-1,1]^3\setminus\Xi$ since $\bXi$ is a boundary link in $[-1,1]^3\setminus\eta$.

The notion of a $q$--tangle in $[-1,1]^3$ with trivial paths, {\em i.e.} line segments, is equivalent to the following one. 
A {\em $q$--tangle with disks} is an equivalence class of pairs $(\gamma,k)$, 
where $\gamma$ is a $q$--tangle in $[-1,1]^3$, $k$ is a non-negative integer understood as the datum of $k$ disks $d_i=[0,1]\times[-1,1]\times\{h_i(k)\}$ and the link $\sqcup_{i=1}^k\partial d_i$ associated with the paths $d_i^\partial=\{0\}\times[-1,1]\times\{h_i(k)\}$ is a boundary link in $[-1,1]^3\setminus\gamma$. An example of such a tangle with disks is drawn in Figure \ref{figtang}, projected in the $y$ direction. 
Equivalence of such pairs is defined as isotopy relative to $(\partial[-1,1]^3)\cup(\cup_{i=1}^k d_i^\partial)$. 
Define two categories $\tang$ and $\ttang$ with objects the non-associative words in the letters $(+,-)$ and morphisms the $q$--tangles for $\tang$ and the 
$q$--tangles with disks for $\ttang$. Composition is given by vertical juxtaposition. Given a $q$--tangle $\gamma$ and a $q$--tangle with disks $(\upsilon,k)$, 
define the tensor product $\gamma\otimes(\upsilon,k)$ in $\ttang((w_t(\gamma))(w_t(\upsilon)),(w_b(\gamma))(w_b(\upsilon)))$ by horizontal juxtaposition in the $x$ direction. 
Define similarly two categories $\mathcal{T}$ and $\widetilde{\mathcal{T}}$ of tangles and tangles with disks in $[-1,1]^3$ with objects the associative words 
in the letters $(+,-)$. 
\begin{figure}[htb] 
\begin{center}
\begin{tikzpicture} [scale=0.4]
 \draw (6,8) .. controls +(0,-2) and +(1,0) .. (5.2,4.7);
 \draww{(3.5,2) .. controls +(0,1) and +(0,-2) .. (6,4) .. controls +(0,0.7) and +(1,0) .. (5.2,5.3);}
 \draw (5.2,5.3) .. controls +(-1,0) and +(1,0) .. (4.7,3);
 \draww{(5.2,4.7) .. controls +(-0.8,0) and +(0,-1) .. (4.5,6) .. controls +(0,1) and +(0,-1) .. (4,8);}
 \draww{(4,0) .. controls +(0,1) and +(0,-1) .. (5.5,2) .. controls +(0,0.5) and +(0.8,0) .. (4.7,3.5);}
 \draww{((4.7,3.5) .. controls +(-0.8,0) and +(0,0.5) .. (4.5,2) .. controls +(0,-1) and +(0,1) .. (2,0);}
 \draww{(4.7,3) .. controls +(-1,0) and +(0,-3) .. (2,8);}
 \draww{(6,0) .. controls +(0,2) and +(0,-1) .. (3.5,2);}
 \draw (0,0) -- (0,8) -- (8,8) -- (8,0) -- (0,0);
 \foreach \y in {2,4,6} {\draw (4,\y) node {$\scriptstyle{\bullet}$} -- (8,\y);}
 \draw[->] (5.8,0.8) -- (5.7,0.9); \draw[->] (4.5,0.83) -- (4.6,0.9); \draw[->] (6.01,6.9) -- (6.01,7);
\end{tikzpicture}
\end{center} \caption{Diagram of a tangle with disks.} \label{figtang}
\end{figure}
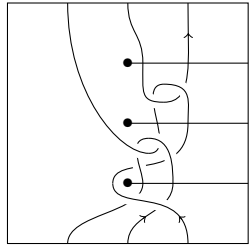

      \section{Target categories: Jacobi diagrams with beads} \label{sectargetcat}

     \subsection{Diagram spaces} \label{subsecdiagramspaces}

For a compact oriented 1--manifold $X$ and a finite set $C$, a {\em Jacobi diagram on $(X,C)$} is a unitrivalent graph whose trivalent vertices are oriented 
and whose univalent vertices are embedded in $X$ or labelled by $C$, where an orientation of a trivalent vertex is a cyclic order of the three edges that 
meet at this vertex ---fixed as \raisebox{-1.5ex}{
\begin{tikzpicture} [scale=0.2]
\newcommand{\tiers}[1]{
\draw[rotate=#1,color=white,line width=4pt] (0,0) -- (0,-2);
\draw[rotate=#1] (0,0) -- (0,-2);}
\draw (0,0) circle (1);
\draw[<-] (-0.05,1) -- (0.05,1);
\tiers{0}
\tiers{120}
\tiers{-120}
\end{tikzpicture}}
in the pictures. The manifold $X$ is called the {\em skeleton} of the diagram. Next, let $R$ be the ring $\Qtt$ or $\Qt$. 
An {\em $R$--beaded Jacobi diagram on $(X,C)$} is a Jacobi diagram on $(X,C)$ whose graph edges are oriented and labelled by $R$. 
Last, an {\em $R$--winding Jacobi diagram on $(X,C)$} is an $R$--beaded Jacobi diagram on $(X,C)$ whose skeleton is viewed as a union of edges ---defined 
by the embedded vertices--- that are labelled by powers of $t$, with the condition that the product of the labels on each component of $X$ is~1. 
As defined in the introduction, the i--degree of a trivalent diagram is its number of trivalent vertices. 
Set:
$$\A(X,*_C)=\frac{\Q\langle \textrm{Jacobi diagrams on }(X,C) \rangle}{\Q\langle \textrm{AS, IHX, STU} \rangle},$$
$$\tA_R(X,*_C)=\frac{\Q\langle R\textrm{--beaded Jacobi diagrams on }(X,C) \rangle}{\Q\langle \textrm{AS, IHX, STU, LE, OR, Hol} \rangle},$$
$$\tAw_R(X,*_C)=\frac{\Q\langle R\textrm{--winding Jacobi diagrams on }(X,C) \rangle}{\Q\langle \textrm{AS, IHX, STU, LE, OR, Hol, \holw} \rangle},$$
with the relations in Figures \ref{figrelations1}, \ref{figrelations2} and \ref{figrelations3}, where the IHX relation for beaded and winding 
diagrams is defined with the central edge labelled by 1. 
\begin{figure}[htb] 
\begin{center}
\begin{tikzpicture} [scale=0.3]
\begin{scope} 
 \draw (0,0) -- (0,4);
 \draw[dashed,->] (0,2) -- (0.8,2) node[above] {1};
 \draw[dashed] (0.7,2) -- (1.3,2) -- (2.6,3) (1.3,2) -- (2.6,1);
 \draw (3.8,2) node{$=$};
 \draw (5,0) -- (5,4);
 \draw[dashed] (5,2.6) -- (7.6,2.6) (5,1.4) -- (7.6,1.4);
 \draw (8.8,2) node {$-$};
 \draw (10,0) -- (10,4);
 \draw[dashed] (10,2.6) -- (12.6,1.4) (10,1.4) -- (12.6,2.6);
 \draw (6.3,-1.5) node{STU};
\end{scope}
\begin{scope} [xshift=20cm]
 \draw[->] (0,0) -- (0,4);
 \draw (0,3) node {$\scriptscriptstyle{\bullet}$} node[left] {$t^i$};
 \draw (0,1) node {$\scriptscriptstyle{\bullet}$} node[left] {$t^j$};
 \draw[dashed] (0,2) -- (2.6,2);
 \draw[dashed,->] (0,2) -- (1.5,2) node[above] {$P$};
 \draw (4,2) node {$=$};
 \draw[->] (8,0) -- (8,4);
 \draw (8,3) node {$\scriptscriptstyle{\bullet}$} node[left] {$t^{i+1}$};
 \draw (8,1) node {$\scriptscriptstyle{\bullet}$} node[left] {$t^{j-1}$};
 \draw[dashed] (8,2) -- (10.6,2);
 \draw[dashed,->] (8,2) -- (9.5,2) node[above] {$tP$};
 \draw (5,-1.5) node {\holw};
\end{scope}
\end{tikzpicture}
\end{center} \caption{Relations STU and \holw\ on Jacobi diagrams.} \label{figrelations3}
\end{figure}
In the STU relation, the edges corresponding to each other have the same orientation and label.
In the pictures, the skeleton is represented with full lines and the graph with dashed lines. We indeed consider the i--degree completion of these vector spaces,
keeping the same notation.

\begin{remark}
 For diagrams in $\tAw_R(X,*_C)$, the condition on the labels on the skeleton implies that all labels can be pushed off each component of the skeleton using the \holw\ relation. When the component is an interval, there is a unique way to do so. Hence, when $X$ contains only intervals, $\tAw_R(X,*_C)$ is isomorphic to $\tA_R(X,*_C)$.
\end{remark}

For a finite set $S$, denote by $\xd_S$ (resp. $\xc_S$) the manifold made of $|S|$ intervals (resp. circles) indexed by the elements of $S$. 
In the following, $\bA$ stands for $\A$, $\tA_R$ or $\tAw_R$. In \cite[Theorem 8]{BN}, Bar-Natan defines a formal PBW isomorphism:
$$\chi_S:\bA(X,*_{C\cup S})\iso\bA(X\cup\xd_S,*_C).$$
For a Jacobi diagram $D$, the image $\chi_S(D)$ is the average of all possible ways to attach the $s$--colored vertices of $D$ on the corresponding $s$--indexed 
interval in $\xd_S$ for each $s\in S$. The setting of \cite{BN} is not exactly the same, but the argument adapts directly. 
When $|S|=1$, closing the $S$--labelled component gives an isomorphism from $\bA(X\cup\xd_S,*_C)$ to $\bA(X\cup\xc_S,*_C)$ \cite[Lemma 3.1]{BN}. However, this isomorphism does not hold for $|S|>1$. To recover an isomorphism onto $\bA(X\cup\xc_S,*_C)$, some ``link relations'' were introduced in \cite[Section 5.2]{AA2}. 
We recall these relations and introduce additional ``winding relations''.

Given a (beaded, winding) Jacobi diagram $D$ on $(X,C\cup S)$ and a univalent vertex $*$ of $D$ labelled by $s\in S$, define the associated {\em link relation} 
as the vanishing of the sum of all diagrams obtained from $D$ by gluing the vertex $*$ on the edges adjacent to a univalent $s$--labelled vertex, as follows: 
\raisebox{-1.5ex}{
\begin{tikzpicture} [scale=0.4]
 \draw[dashed] (0,1) -- (1,1) (1,0) -- (1,2);
 \draw (1,1) node {$\scriptscriptstyle{\bullet}$};
 \draw (1,2) node {$\scriptscriptstyle{\bullet}$};
 \draw (1,1) node[right] {$*$} (1,2) node[right] {$s$};
\end{tikzpicture}},
see Figure \ref{figlinkrel} (we omit the orientation of the edges when it is not relevant thanks to the OR relation). 
\begin{figure}[htb] 
\begin{center}
\begin{tikzpicture} [scale=0.45]
\begin{scope}
 \draw[->] (0,0) -- (0,4);
 \draw[dashed] (0,2) -- (1,2) -- (2.5,1) (1.75,2.5) -- (2.5,3);
 \draw[->,dashed] (1,2) -- (1.75,2.5);
 \draw (1.6,2.4) node[above] {$\scriptstyle{t}$};
 \draw (2.5,3) node[right] {$\scriptstyle{s}$} (2.5,1) node[right] {$\scriptstyle{s'}$};
 \draw (6,2) circle (1);
 \draw[->] (7,2) -- (7,2.05);
 \draw[dashed] (6,3) -- (6,4) (5.25,0.66) -- (5,0.2) (6.5,1.12) -- (7,0.2);
 \draw (6,4) node {$*$};
 \draw (6,4) node[right] {$\scriptstyle{s}$} (5,0.2) node[right] {$\scriptstyle{s'}$} (7,0.2) node[right] {$\scriptstyle{s}$};
 \draw[->,dashed] (5.5,1.12) -- (5.25,0.66);
 \draw (5.4,0.8) node[left] {$\scriptstyle{t^2}$};
\end{scope}
 \draw (10,2) node {$\rightsquigarrow$};
\begin{scope} [xshift=12cm]
 \draw[->] (0,0) -- (0,4);
 \draw[dashed] (0,2) -- (1,2) -- (2.5,1) (1.5,2.33) -- (2.5,3);
 \draw[->,dashed] (1,2) -- (1.5,2.33);
 \draw (1.25,2.6) node {$\scriptstyle{t}$};
 \draw (2.5,3) node[right] {$\scriptstyle{s}$} (2.5,1) node[right] {$\scriptstyle{s'}$};
 \draw (6,2) circle (1);
 \draw[->] (7,2) -- (7,2.05);
 \draw[dashed] (5.25,0.66) -- (5,0.2) (6.5,1.12) -- (7,0.2);
 \draw[dashed] (6,3) .. controls +(0,2) and +(-1,1.5) .. (2,2.66);
 \draw (5,0.2) node[right] {$\scriptstyle{s'}$} (7,0.2) node[right] {$\scriptstyle{s}$};
 \draw[->,dashed] (5.5,1.12) -- (5.25,0.66);
 \draw (5.4,0.8) node[left] {$\scriptstyle{t^2}$};
\end{scope}
 \draw (21,2) node {$+$};
\begin{scope} [xshift=23cm]
 \draw[->] (0,0) -- (0,4);
 \draw[dashed] (0,2) -- (1,2) -- (2.5,1) (1.75,2.5) -- (2.5,3);
 \draw[->,dashed] (1,2) -- (1.75,2.5);
 \draw (1.6,2.4) node[above] {$\scriptstyle{t}$};
 \draw (2.5,3) node[right] {$\scriptstyle{s}$} (2.5,1) node[right] {$\scriptstyle{s'}$};
 \draw (6,2) circle (1);
 \draw[->] (7,2) -- (7,2.05);
 \draw[dashed] (5.25,0.66) -- (5,0.2) (6.5,1.12) -- (7,0.2);
 \draw[dashed] (6,3) .. controls +(0,2) and +(1.84,1) .. (6.75,0.66);
 \draw (5,0.2) node[right] {$\scriptstyle{s'}$} (7,0.2) node[right] {$\scriptstyle{s}$};
 \draw[->,dashed] (5.5,1.12) -- (5.25,0.66);
 \draw (5.4,0.8) node[left] {$\scriptstyle{t^2}$};
\end{scope}
 \draw (32,2) node {$=$} (33.5,2.06) node {$0$};
\end{tikzpicture}
\end{center} \caption{A link relation.} \label{figlinkrel}
\end{figure}
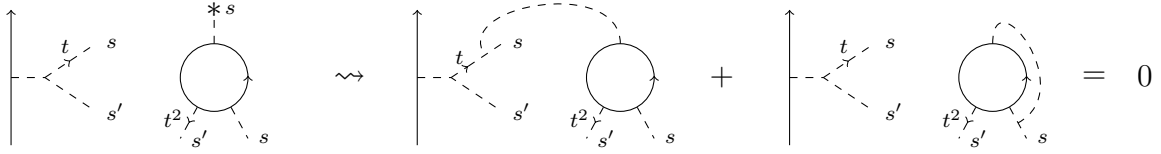
Given a winding Jacobi diagram $D$ on $(X,C\cup S)$, a label $s\in S$ and an integer $k$, the associated {\em winding relation} identifies $D$ 
with the diagram obtained from $D$ by {\em pushing $t^k$ at each $s$--labelled vertex}, {\em i.e.} by multiplying the label of each edge adjacent 
to a univalent $s$--labelled vertex by $t^k$ if the orientation of the edge goes backward the vertex and by $t^{-k}$ otherwise, see Figure~\ref{figwindingrel}. 
\begin{figure}[htb] 
\begin{center}
\begin{tikzpicture} [scale=0.45]
\begin{scope}
 \draw[dashed,->] (-2,1) -- (-2,2) node[right] {$\scriptstyle{t^{-1}}$};
 \draw[dashed] (-2,2) -- (-2,3);
 \draw (-1.8,0.8) node {$\scriptstyle{s}$} (-1.8,3.2) node {$\scriptstyle{s}$};
 \draw[->] (0,0) -- (0,4);
 \draw[dashed] (0,1) -- (1,1) -- (2.5,0) (1,1) -- (1.75,1.5);
 \draw[->,dashed] (2.5,2) -- (1.75,1.5);
 \draw (1.6,1.4) node[above] {$\scriptstyle{t^3}$};
 \draw (2.7,2) node {$\scriptstyle{s}$} (2.7,0) node {$\scriptstyle{s'}$};
 \draw[dashed,->] (0,3) -- (1.4,3);
 \draw[dashed] (1.4,3) -- (2.3,3);
 \draw (1.4,3.5) node {$\scriptstyle{t^2}$};
 \draw (2.7,3) node {$\scriptstyle{s}$};
\end{scope}
 \draw (5,2) node {$=$};
\begin{scope} [xshift=9cm]
 \draw[dashed,->] (-2,1) -- (-2,2) node[right] {$\scriptstyle{t^{-1}}$};
 \draw[dashed] (-2,2) -- (-2,3);
 \draw (-1.8,0.8) node {$\scriptstyle{s}$} (-1.8,3.2) node {$\scriptstyle{s}$};
 \draw[->] (0,0) -- (0,4);
 \draw[dashed] (0,1) -- (1,1) -- (2.5,0) (1,1) -- (1.75,1.5);
 \draw[->,dashed] (2.5,2) -- (1.75,1.5);
 \draw (1.6,1.4) node[above] {$\scriptstyle{t^4}$};
 \draw (2.7,2) node {$\scriptstyle{s}$} (2.7,0) node {$\scriptstyle{s'}$};
 \draw[dashed,->] (0,3) -- (1.4,3);
 \draw[dashed] (1.4,3) -- (2.3,3);
 \draw (1.4,3.5) node {$\scriptstyle{t}$};
 \draw (2.7,3) node {$\scriptstyle{s}$};
\end{scope}
\end{tikzpicture}
\end{center} \caption{A winding relation.} \label{figwindingrel}
\end{figure}
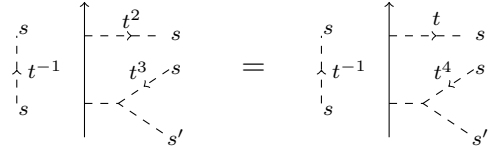
Denote $\bA(X,*_C,\xl_S)$ (resp. $\bA(X,*_C,\xlw_S)$) the quotient of $\bA(X,*_{C\cup S})$ by all link relations (resp. all link and winding relations) 
on $S$--labelled vertices. Note that if $X$ contains no closed component, then the spaces $\tAw_R(X,*_C,\xlw_S)$ and $\tA_R(X,*_C,\xlw_S)$ are isomorphic. 
When some of the sets $X$, $C$, $S$ are empty, we simply drop the corresponding notation, mentionning $\varnothing$ only when they are all empty. 

\begin{proposition}
 The isomorphisms $\chi_S:\bA(X,*_{C\cup S})\iso\bA(X\cup\xd_S,*_C)$ descend to isomorphisms: 
 $$\chi_S:\A(X,*_C,\xl_S)\iso\A(X\cup\xc_S,*_C),$$
 $$\chi_S:\tA_R(X,*_C,\xl_S)\iso\tA_R(X\cup\xc_S,*_C),$$
 $$\chi_S:\tAw_R(X,*_C,\xlw_S)\iso\tAw_R(X\cup\xc_S,*_C).$$
\end{proposition}
\begin{proof}
 In the case $\bA=\A\textrm{ or }\tA_R$, it is \cite[Theorem 3]{AA2}. We recall briefly their argument in order to add the consideration of the winding relations 
 when $\bA=\tAw_R$. 
 
 The fact that the images by $\chi_S$ of the link relations map to 0 in $\bA(X\cup\xc_S,*_C)$ follows from the STU relation. For the winding relations, it follows 
 from the \holw\ relation applied at each univalent vertex glued on the $s$--labelled component, where $s$ is the label involved in the relation. 
 Now, take two diagrams in $\bA(X\cup\xd_S,*_C)$ that are identified when closing an $S$--labelled component. We have to consider the two situations 
 depicted in Figures~\ref{figinverselinkrel} and~\ref{figinversewindingrel}, where the gray zone represents a hidden part of the diagram. 
\begin{figure}[htb] 
\begin{center}
\begin{tikzpicture} [scale=0.4]
\begin{scope}
 \draw[->] (3,0) -- (3,5);
 \foreach \x in {1,2,3,4} {\draw[dashed] (0,\x) -- (3,\x);}
 \draw[gray,fill=gray] (0,2.5) circle (1 and 2.5);
\end{scope}
 \draw (5,2.5) node {$-$};
\begin{scope} [xshift=8cm]
 \draw[->] (3,0) -- (3,5);
 \foreach \x in {1,2,3} {\draw[dashed] (0,\x) -- (3,\x);}
 \draw[dashed] (0.5,4) .. controls +(1.5,0) and +(-1.5,0) .. (3,0.3);
 \draw[gray,fill=gray] (0,2.5) circle (1 and 2.5);
\end{scope}
 \draw (13,2.5) node {$=$};
\begin{scope} [xshift=16cm]
 \draw[->] (3,0) -- (3,5);
 \foreach \x in {1,2,3} {\draw[dashed] (0,\x) -- (3,\x);}
 \draw[dashed] (0.5,4) .. controls +(1,0) and +(0,0.5) .. (2,3);
 \draw[gray,fill=gray] (0,2.5) circle (1 and 2.5);
\end{scope}
 \draw (21,2.5) node {$+$};
\begin{scope} [xshift=24cm]
 \draw[->] (3,0) -- (3,5);
 \foreach \x in {1,2,3} {\draw[dashed] (0,\x) -- (3,\x);}
 \draw[dashed] (0.5,4) .. controls +(1.3,0) and +(0,0.6) .. (2,2);
 \draw[gray,fill=gray] (0,2.5) circle (1 and 2.5);
\end{scope}
 \draw (29,2.5) node {$+$};
\begin{scope} [xshift=32cm]
 \draw[->] (3,0) -- (3,5);
 \foreach \x in {1,2,3} {\draw[dashed] (0,\x) -- (3,\x);}
 \draw[dashed] (0.5,4) .. controls +(1.5,0) and +(0,0.8) .. (2,1);
 \draw[gray,fill=gray] (0,2.5) circle (1 and 2.5);
\end{scope}
\end{tikzpicture}
\end{center} \caption{Recovering the link relation.} \label{figinverselinkrel}
\end{figure}
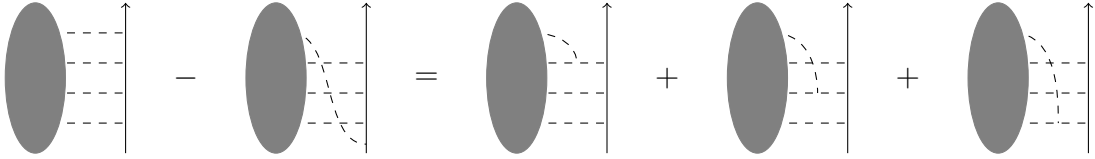
Equalities are obtained by applying STU relations in the first case and \holw\ relations in the second case.
\begin{figure}[htb] 
\begin{center}
\begin{tikzpicture} [scale=0.35]
\begin{scope}
 \draw[->] (4,0) -- (4,8);
 \foreach \x in {2,4,6} {\draw[dashed] (0,\x) -- (4,\x);}
 \foreach \x in {3,5,7} {\draw (4,\x) node {$\scriptscriptstyle{\bullet}$};}
 \draw (4.7,3) node {$\scriptstyle{t^j}$} (4.7,5) node {$\scriptstyle{t^\ell}$} (4.7,7) node {$\scriptstyle{t^m}$};
 \draw[gray,fill=gray] (0,4) circle (1 and 3);
\end{scope}
 \draw (7,4) node {$-$};
\begin{scope} [xshift=10cm]
 \draw[->] (4,0) -- (4,8);
 \foreach \x in {2,4,6} {\draw[dashed] (0,\x) -- (4,\x);}
 \foreach \x in {1,3,5} {\draw (4,\x) node {$\scriptscriptstyle{\bullet}$};}
 \draw (4.7,3) node {$\scriptstyle{t^j}$} (4.7,5) node {$\scriptstyle{t^\ell}$} (4.7,1) node {$\scriptstyle{t^m}$};
 \draw[gray,fill=gray] (0,4) circle (1 and 3);
\end{scope}
 \draw (17,4) node {$=$};
\begin{scope} [xshift=20cm]
 \draw[->] (4,0) -- (4,8);
 \foreach \x in {2,4,6} {\draw[dashed] (0,\x) -- (4,\x);}
 \foreach \x in {4,6} {\draw[->] (2.5,\x) -- (2.45,\x);}
 \draw (2.7,4.6) node {$\scriptstyle{t^j}$} (2.9,6.6) node {$\scriptstyle{t^{j+\ell}}$};
 \draw[gray,fill=gray] (0,4) circle (1 and 3);
\end{scope}
 \draw (27,4) node {$-$};
\begin{scope} [xshift=30cm]
 \draw[->] (4,0) -- (4,8);
 \foreach \x in {2,4,6} {\draw[dashed] (0,\x) -- (4,\x);}
 \foreach \x in {2,4} {\draw[->] (2.5,\x) -- (2.45,\x);}
 \draw (2.7,2.6) node {$\scriptstyle{t^m}$} (2.9,4.6) node {$\scriptstyle{t^{j+m}}$};
 \draw[gray,fill=gray] (0,4) circle (1 and 3);
\end{scope}
\end{tikzpicture}
\end{center} \caption{Recovering the winding relation (with $j+\ell+m=0$).} \label{figinversewindingrel}
\end{figure}
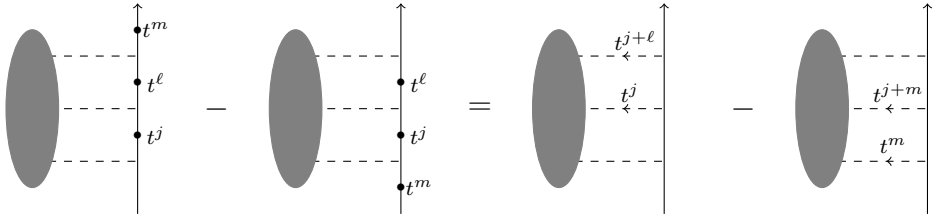
 Application of $\chi^{-1}$ to the right members 
 provides linear combinations of the same sum with the skeleton component dropped and possibly trees glued. Using the IHX relation in the first case 
 (resp. the Hol relation in the second case), we obtain link relations (resp. winding relations).
\end{proof}

    \subsection{Product and coproduct}

We first define a coproduct on the diagram spaces of the previous subsection. Given a (beaded, winding) Jacobi diagram $D$ on $(X,C)$, denote by $\dddot{D}$ its graph part, and by $\dddot{D}_i$, $i\in I$, the connected components of $\dddot{D}$. Set $D_J=D\setminus(\sqcup_{i\in I\setminus J}\dddot{D}_i)$. In the winding case, multiply the labels of the concatenated edges of the skeleton. Define the coproduct of a diagram $D$ by $$\Delta(D)=\sum_{J\subset I}D_J\otimes D_{I\setminus J}.$$
Note that the different relations on Jacobi diagrams respect the coproduct. This provides a notion of {\em group-like} elements, {\em i.e.} elements $G$ such that $\Delta(G)=G\otimes G$. 

Set $\bA=\A\textrm{ or }\tA_R$. We will define a Hopf algebra structure on $\bA(*_C)$. Define the product of two diagrams as the disjoint union. 
The unit $\epsilon:\Q\to\bA(*_C)$ is defined by $\epsilon(1)=\varnothing$ and the counit $\varepsilon:\bA(*_C)\to\Q$ is given by 
$\varepsilon(D)=0$ if $D\neq\varnothing$ and $\varepsilon(\varnothing)=1$. The antipode is given by $D\mapsto(-1)^{|I|}D$. We finally have a structure of a 
graded Hopf algebra on $\bA(*_C)$, where the grading is given by the i--degree. It is known that an element in a graded Hopf algebra is group-like if and only if it is the exponential of a {\em primitive} element, 
{\em i.e.} an element $G$ such that $\Delta(G)=1\otimes G+G\otimes 1$. Here, the primitive elements are the series of connected diagrams. 

The isomorphisms $\chi$ of the previous subsection are not algebra morphisms, but they preserve the coproduct. For the spaces $\bA(*_C)$ with $\bA=\A\textrm{ or }\tA_R$, we have an exponential map associated with the product. For general (beaded, winding) Jacobi diagrams, we will use the notation $\expd$, namely exponential with respect to the disjoint union, for linear combination of diagrams with no univalent vertex embedded in the skeleton, where the disjoint union applies only to the graph part.

    \subsection{Formal Gaussian integration}

This part aims at defining a formal Gaussian integration along $S$ on $\tAw_R(X,*_{C\cup S})$. 
\begin{definition}
 A (beaded, winding) Jacobi diagram on $(X,C\cup S)$ is {\em substantial} if it has no {\em strut}, {\em i.e.} no isolated dashed edge. 
 It is {\em $S$--substantial} if it has no {\em $S$--strut}, {\em i.e.} no strut with both vertices labelled in $S$.
\end{definition}

Given two (beaded, winding) Jacobi diagrams $D$ and $E$ on $(X,C\cup S)$, one of whose is $S$--substantial, define $\langle D,E\rangle_S$ as the sum 
of all diagrams obtained by gluing all $s$--colored vertices of $D$ with all $s$--colored vertices of $E$ for all $s\in S$ ---if the numbers of $s$--colored vertices in $D$ and $E$ do not match for some $s\in S$, then $\langle D,E\rangle_S=0$. In the beaded and winding cases, we must precise the orientation 
and label of the created edges. Such an edge is the gluing of two or three edges in the initial diagrams. Fix arbitrarily the orientation of the new edge. 
Let $P(t)$ (resp. $Q(t)$) be the product of the labels of the initial edges whose orientation coincides (resp. does not coincide). Define the label of the new edge 
as $P(t)Q(t^{-1})$, see Figure \ref{figbracket}. 
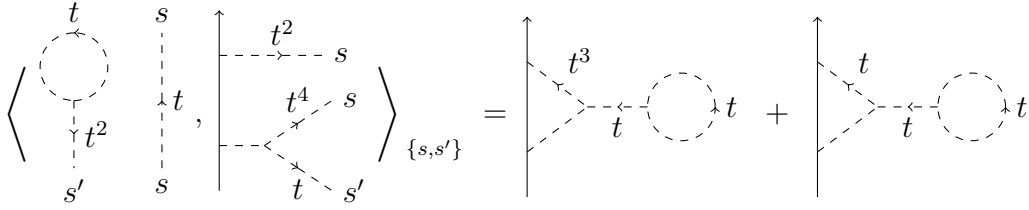
\begin{figure} [htb]
$$\left\langle\raisebox{-1.2cm}{
\begin{tikzpicture} [scale=0.9]
\begin{scope}
 \draw[dashed,->] (0,1) -- (0,0.5) node[right] {$t^2$}; \draw[dashed] (0,0.5) -- (0,0); 
 \draw[dashed,->] (0,1) arc (-90:90:0.5); \draw[dashed] (0,2) arc (90:270:0.5);
 \draw (0,0) node[below] {$s'$}; \draw (0,2) node[above] {$t$};
\end{scope}
\begin{scope} [xshift=1.3cm]
 \draw[dashed,->] (0,0) -- (0,1) node[right] {$t$}; \draw[dashed] (0,1) -- (0,2);
 \draw (0,0) node[below] {$s$}; \draw (0,2) node[above] {$s$};
\end{scope}
\end{tikzpicture}}
,\raisebox{-1.2cm}{
\begin{tikzpicture} [scale=0.6]
 \draw[->] (0,0) -- (0,4);
 \draw[dashed] (0,1) -- (1,1) (1.75,0.5) -- (2.5,0) node[right] {$s'$} (1.75,1.5) -- (2.5,2) node[right] {$s$};
 \draw[->,dashed] (1,1) -- (1.75,1.5) node[above] {$t^4$};
 \draw[->,dashed] (1,1) -- (1.75,0.5) node[below] {$t$};
 \draw[dashed,->] (0,3) -- (1.4,3);
 \draw[dashed] (1.4,3) -- (2.3,3);
 \draw (1.4,3.5) node {$t^2$};
 \draw (2.7,3) node {$s$};
\end{tikzpicture}}
\right\rangle\raisebox{-2ex}{${}_{\{s,s'\}}$}\ =\raisebox{-1cm}{
\begin{tikzpicture} 
 \draw[->] (0,0) -- (0,2.4);
 \begin{scope} [xscale=0.4,yscale=0.6]
  \draw[dashed,->] (0,3) -- (1,2.5) (0,1) -- (2,2) -- (1,2.5) node[above right] {$t^3$};
  \draw[dashed,->] (2,2) -- (3,2) (4,2) -- (3,2) node[below] {$t$};
 \end{scope}
 \draw[dashed,->] (2.5,1.2) arc (0:180:0.45) arc (180:360:0.45) node[right] {$t$};
\end{tikzpicture}}
\ +\ \raisebox{-1cm}{
\begin{tikzpicture} 
 \draw[->] (0,0) -- (0,2.4);
 \begin{scope} [xscale=0.4,yscale=0.6]
  \draw[dashed,->] (0,3) -- (1,2.5) (0,1) -- (2,2) -- (1,2.5) node[above right] {$t$};
  \draw[dashed,->] (2,2) -- (3,2) (4,2) -- (3,2) node[below] {$t$};
 \end{scope}
 \draw[dashed,->] (2.5,1.2) arc (0:180:0.45) arc (180:360:0.45) node[right] {$t$};
\end{tikzpicture}}$$
\caption{Bracketting diagrams.} \label{figbracket}
\end{figure}
We have the following immediate lemma.
\begin{lemma} \label{lemmawdrel}
 If $D'$ and $E'$ are obtained from $D$ and $E$ by applying the same winding relation on $s$--labelled vertices for some $s\in S$, 
 then $\langle D,E\rangle_S=\langle D',E'\rangle_S$.
\end{lemma}
This bracketting defines a $\Q$--bilinear operator on the diagram spaces $\bA(X,*_{C\cup S})$ for $\bA=\A,\ \tA_R,\textrm{ or }\tAw_R$. 
\begin{theorem}[Jackson, Moffatt, Morales] \label{thJMM}
 Assume the $1$--manifold $X$ is a disjoint union of intervals. If $G$ and $H$ are group-like in $\bA(X,*_{C\cup S})$, then $\langle G,H\rangle_S$ is also group-like.
\end{theorem}
\begin{proof}
 When $X=\varnothing$, it is \cite[Theorem 2.4]{JMM}. The case of a non-empty $X$ follows since the isomorphism $\chi_{\pi_0(X)}:\bA(*_{\pi_0(X)\cup C\cup S})\iso\bA(X,*_{C\cup S})$ preserves the coproduct and the bracketting $\langle\cdot,\cdot\rangle_S$.
\end{proof}

\begin{notation}
 If $W=(W_{ij}(t))_{i,j\in S}$ is an $(S,S)$--matrix with coefficients in $\Qt$, we also denote $W=\sum_{i,j\in S}$\raisebox{-2.3ex}{
 \begin{tikzpicture} [scale=0.3]
 \draw[dashed,->] (0,0) -- (0,1);
 \draw[dashed] (0,1) -- (0,2);
 \draw (0.4,-0.2) node {$\scriptscriptstyle{i}$} (0.4,2.2) node {$\scriptscriptstyle{j}$};
 \draw (1.7,0.8) node {$\scriptscriptstyle{W_{ij}(t)}$};
 \end{tikzpicture}}.
\end{notation}

\begin{definition}
 An element $G\in\tAw_{\Qtt}(X,*_{C\cup S})$ is {\em Gaussian} if $G=\expd(\frac{1}{2}W(t))\sqcup H$ where $W(t)$ is an $(S,S)$--matrix with coefficients in $\Qtt$ 
 and $H$ is $S$--substantial. If $\det(W(t))\neq0$, $G$ is {\em non-degenerate} and we set:
 $$\int_S G=\langle\,\expd(-\frac{1}{2}W^{-1}(t)),H\,\rangle_S\ \in\tAw_{\Qt}(X,*_C).$$
\end{definition}
\begin{lemma} \label{lemmaFGI}
 Let $G=\expd(\frac{1}{2}W(t))\sqcup H$ be a non-degenerate Gaussian in $\tAw_{\Qtt}(X,*_{C\cup S})$.
 \begin{itemize}
  \item If a non-degenerate Gaussian $\expd(\frac{1}{2}W(t))\sqcup H'$ is equal to $G$ in $\tAw_{\Qtt}(X,*_C,\xl_S)$, then 
   $\int_S(\expd(\frac{1}{2}W(t))\sqcup H')=\int_S G$.
  \item If $G'=\expd(\frac{1}{2}W'(t))\sqcup H'$ is obtained from $G=\expd(\frac{1}{2}W(t))\sqcup H$ by applying a winding relation, then $\int_S G'=\int_S G$.
 \end{itemize}
\end{lemma}
\begin{proof}
 The first point is essentially given by the proof of Bar-Natan and Lawrence \cite[Proposition 2.2]{BNL} in the non-beaded case. Here, multiplication 
 by \raisebox{-0.64cm}{
 \begin{tikzpicture} [scale=0.4]
  \draw[dashed,->] (0,0) -- (0,1) node[right] {$\scriptstyle{W_{ss}(t)}$}; \draw[dashed] (0,1) -- (0,2);
  \draw (0.2,-0.2) node {$\scriptstyle{s}$}; \draw (0.2,2.1) node {$\scriptstyle{s}$}; 
 \end{tikzpicture}}
 does not preserve the link relations, but the supplementary terms vanish thanks to the AS relation when applying $\langle\,\cdot\,,\expd(-\frac{1}{2}
 \raisebox{-0.64cm}{
 \begin{tikzpicture} [scale=0.4]
  \draw[dashed,->] (0,0) -- (0,1) node[right] {$\scriptstyle{W_{ss}(t)}$}; \draw[dashed] (0,1) -- (0,2);
  \draw (0.2,-0.2) node {$\scriptstyle{s}$}; \draw (0.2,2.1) node {$\scriptstyle{s}$}; 
 \end{tikzpicture}})\rangle$.
 
 The second point follows from Lemma \ref{lemmawdrel}.
\end{proof}

    \subsection{Categories of diagrams}

For $\bA=\A,\ \tA_R,\textrm{ or }\tAw_R$, define a category $\bA$ whose objects are associative words in the letters $(+,-)$ and whose set of morphisms are 
$\bA(v,u)=\oplus_X\bA(X)$, where $X$ runs over all compact oriented 1--manifolds with boundary identified with the set of letters of $u$ and $v$, 
with the following sign convention: for $u$, a $+$ when the orientation of $X$ goes towards the boundary point and a $-$ when it goes backward, 
and the converse for $v$. Composition is given by vertical juxtaposition, where the label of the created edges in the case of beaded or winding diagrams 
is defined with the same rule as in the definition of $\langle D,E\rangle$. The tensor product given by disjoint union defines a strict 
monoidal structure on $\bA$. 

We finally define the target category of our extended Kricker invariant. 
\begin{notation}
 Given a positive integer $g$ and a symbol $\natural$, set $\lfloor g\rceil^\natural=\{1^\natural,\dots,g^\natural\}$. Set $\lfloor 0\rceil^\natural=\varnothing$. 
\end{notation}
\begin{definition}
 Fix non-negative integers $f$ and $g$. An $R$--beaded Jacobi diagram on $(\varnothing,\lfloor g\rceil^+\cup\lfloor f\rceil^-)$ is {\em top--substantial} 
 if it is $\lfloor g\rceil^+$--substantial.
\end{definition}
Given two such diagrams $D$ and $E$, define their composition $D\circ E$ as the sum of all ways of gluing all $i^+$--labelled vertices of $D$ with all 
$i^-$--labelled vertices of $E$, fixing the orientations and labels of the created edges as in the definition of $\langle D,E\rangle_S$. 
We get a category $\tAts$ whose objects are non-negative integers, with set of morphisms from $g$ to $f$ the subspace of $\tA_{\Qt}(\pgf{g}{f})$ generated 
by top-substantial diagrams. 
The identity of $g$ is $\expd(\sum_{i=1}^g\raisebox{-2.3ex}{
 \begin{tikzpicture} [scale=0.3]
 \draw[dashed] (0,0) -- (0,2);
 \draw (0,0) node[right] {$\scriptstyle{i^-}$} (0,2) node[right] {$\scriptstyle{i^+}$};
 \end{tikzpicture}})$ 
 ---the sum is trivial if $g=0$. 
The tensor product defined by disjoint union of diagrams provides $\tAts$ a strict monoidal structure.

      \section{Winding matrices} \label{secwinding}

In this section, we define winding matrices associated with tangles with disks and bottom-top tangles with paths and interpret them as equivariant linking matrices 
in the case of bottom-top tangles. They will be useful in expressions of our invariant and splitting formulas. 

    \subsection{First definition}

We first define winding matrices in tangles with disks. 
Let $(\gamma,k)$ be a tangle with disks $d_\ell$. Write $\gamma$ as the disjoint union of components $\gamma_i$ for $i=1,\dots,n$. Fix a diagram of $(\gamma,k)$ and 
a base point $\star_i$ for each closed component $\gamma_i$ far from the crossings and the disks. Define the associated {\em winding} $w(\gamma_i,\gamma_j)\in\Ztt$ of 
$\gamma_i$ and $\gamma_j$ in the following way. For a crossing $c$ between $\gamma_i$ and $\gamma_j$, denote $\varepsilon_{ij}(c)$ the algebraic intersection number 
of the union of the disks $d_\ell$ with the path that goes from $\star_i$, or the origin of $\gamma_i$, to $c$ along $\gamma_i$ and then from $c$ to $\star_j$, 
or the end-point of $\gamma_j$, along $\gamma_j$. If $i=j$, change component at the first occurence of $c$. Set
$$w(\gamma_i,\gamma_j)=\left\lbrace\begin{array}{l l}
                         \displaystyle\frac{1}{2}\sum_c \textrm{sg}(c) t^{\varepsilon_{ij}(c)} & \textrm{ if } i\neq j \\ & \\
                         \displaystyle\frac{1}{2}\sum_c \textrm{sg}(c) (t^{\varepsilon_{ii}(c)}+t^{-\varepsilon_{ii}(c)}) & \textrm{ if } i=j
                        \end{array}\right.$$
where the sums are over all crossings between $\gamma_i$ and $\gamma_j$. Note that $w(\gamma_j,\gamma_i)(t)=w(\gamma_i,\gamma_j)(t^{-1})$. 
Now let $I$ and $J$ be two subsets of $\{1,\dots,n\}$ and denote by $\gamma_I$ and $\gamma_J$ the corresponding subtangles of $\gamma$. 
The {\em winding matrix} $W_{\gamma_I\gamma_J}$ associated with the fixed diagram and base points is the matrix whose coefficients are the windings $w(\gamma_i,\gamma_j)$ 
for $i\in I$ and $j\in J$ ---denote it $W_{\gamma_I}$ when $I=J$. In this latter case, note that $W_{\gamma_I}$ is hermitian. 
\begin{lemma} \label{lemmawdmatrix1}
 The winding matrix is invariant by isotopies which do not allow the base points to pass through the disks of the tangle. In particular, when $\gamma$ 
 contains no closed components, it is an isotopy invariant. 
\end{lemma}
\begin{proof}
 First note that the winding matrix is preserved when a crossing passes through a disk $d_\ell$. It is also preserved when a base point of a closed component passes 
 through a crossing since the algebraic intersection number of this component with the union of the disks $d_\ell$ is trivial. Hence it only remains to check invariance with 
 respect to framed Reidemeister moves performed far from the base points and the disks, which is direct. 
\end{proof}
To completely understand the effect of an isotopy on the winding matrix $W_{\gamma_I,\gamma_J}$, we shall describe its modification when a base point passes through 
a disk of the tangle. Fix a closed component $\gamma_i$. Fix a diagram of $(\gamma,k)$ with the base point of $\gamma_i$ located ``just before'' a disk $d_\ell$ of the tangle, 
as shown in the first part of Figure \ref{figbasepoint}.
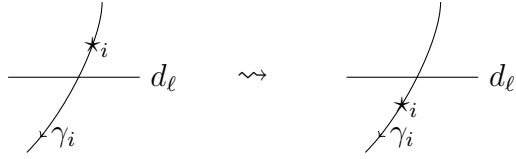
\begin{figure}[htb] 
\begin{center}
\begin{tikzpicture} [scale=0.5]
\begin{scope}
 \draw (0.5,0) -- (4,0) node[right] {$d_\ell$};
 \draw (3,2) .. controls +(0,-1) and +(1,1) .. (1,-2);
 \draw[->] (1.44,-1.5) -- (1.36,-1.6) node[right] {$\gamma_i$};
 \draw (2.87,0.8) node {$\star_i$};
\end{scope}
 \draw (7,0) node {$\rightsquigarrow$};
\begin{scope} [xshift=9cm]
 \draw (0.5,0) -- (4,0) node[right] {$d_\ell$};
 \draw (3,2) .. controls +(0,-1) and +(1,1) .. (1,-2);
 \draw[->] (1.44,-1.5) -- (1.36,-1.6) node[right] {$\gamma_i$};
 \draw (2.09,-0.8) node {$\star_i$};
\end{scope}
\end{tikzpicture}
\caption{Base point passing through a disk.} \label{figbasepoint}
\end{center}
\end{figure}
Consider another diagram of $(\gamma,k)$ which differs from the previous one only by the position of the base point $\star_i$, which is as shown 
on the second part of Figure \ref{figbasepoint}. Let $\varepsilon=\pm1$ give the sign of the intersection of $d_\ell$ and $\gamma_i$ which the base point 
passes through. It is easily seen that the winding matrix of the latter diagram is obtained from the winding matrix of the previous one by multiplication 
on the left by $T_i(t^{-\varepsilon})$ if $i\in I$ and on the right by $T_i(t^{\varepsilon})$ if $i\in J$, where $T_i(t)$ is the diagonal matrix whose 
diagonal coefficients are all 1 except a $t$ at the $i^{th}$ position.

We now define winding matrices for bottom-top tangles in $\Q$--cubes. 
Let $(B,K,\gamma)$ be a bottom-top tangle with paths in a $\Q$--cube. Let $(([-1,1]^3,\Xi,\eta),L)$ be a surgery presentation of $(B,K,\gamma)$. 
Denote $(\eta,k)$ the associated tangle with disks. Fix a diagram of $([-1,1]^3,\Xi,\eta)$ and a base point $\star_i$ for each component $L_i$ of $L$. 
Define the {\em winding matrix of $(B,K,\gamma)$}, with coefficients in $\Qt$, as:
$$W_\gamma=W_\eta-W_{\eta L}W_L^{-1}W_{L\eta}.$$
Note that $W_L(1)$ is the linking matrix of $L$, thus $W_L(1)\neq0$ since $B$ is a homology cube. Hence $W_L$ is invertible over $\Qt$. 
\begin{lemma}
 The winding matrix $W_\gamma$ is an isotopy invariant of $(B,K,\gamma)$.
\end{lemma}
\begin{proof}
 First, when the surgery presentation is fixed, the discussion of the previous subsection implies that the winding matrix does not depend on the choice 
 of diagram and base points. Then independance with respect to Kirby moves is easily checked. We detail the less direct, which is the case when a surgery 
 link component is added to a tangle component. Denote $\eta'$ the tangle obtained from $\eta$ by adding the surgery component $L_j$ to $\eta_i$. 
 We have:
 $$W_{\eta'}=W_\eta+{}^tPW_{L\eta}+W_{\eta L}P+{}^tPW_LP \quad \textrm{and} \quad W_{L\eta'}=W_{L\eta}+W_LP,$$
 where $P$ is the $|L|\times|\eta|$ matrix whose only non-trivial term is a 1 at the $(j,i)$ position. 
\end{proof}

    \subsection{Topological interpretation}

Given a bottom-top tangle with paths in a $\Q$--cube, or a tangle with disks defined from a bottom-top tangle with paths in $[-1,1]^3$ with a surgery link, 
close the components of the tangle, say $\gamma$, by line segments on the top and bottom surfaces to get a link $\bar{\gamma}$. This provides well-defined linking numbers for the tangle components. 
If there is no path or disk, the winding matrix is the linking matrix. It is clear when working in $[-1,1]^3$. For bottom-top tangles in $\Q$--cubes, apply the following 
easy fact, which can be proved by adapting the proof of Proposition~\ref{propintertopo}. 
\begin{fact} \label{fact3.18}
 Let $L$ be an oriented framed link in $[-1,1]^3$ whose linking matrix $\Lk(L)$ is non-degenerate. Let $\xi$ and $\zeta$ be disjoint oriented knots
 in $[-1,1]^3\setminus L$ and denote $\xi',\zeta'$ the copies of $\xi,\zeta$ in the $\Q$--cube obtained from $[-1,1]^3$ by surgery on $L$. Then:
 $$\lk(\xi',\zeta')=\lk(\xi,\zeta)-\Lk(\xi,L).\Lk(L)^{-1}.\Lk(L,\zeta).$$  
\end{fact}
More generally, when there are paths or disks, the winding matrix evaluated at $t=1$ is the linking matrix. We shall give a similar interpretation 
for the winding matrix at a generic $t$. 

Let $(B,K,\gamma)$ be a bottom-top tangle in a $\Q$--cube. Let $E$ be the exterior of $K$ in $B$ and let $\tilde{E}$ be its covering associated with the kernel of the map 
$\pi_1(E)\to\Z=\langle t\rangle$ which sends the positive meridians of the components of $K$ to $t$. The automorphism group of the covering is isomorphic 
to $\Z$; let $\tau$ be the generator associated with the action of the positive meridians. Let $\zeta$ be a knot in $\tilde{E}$ such that there are 
a rational 2--chain $\Sigma$ in $\tilde{E}$ and $P\in\Qtt$ which satisfy $\partial\Sigma=P(\tau)(\zeta)$. 
Let $\xi$ be another knot in $\tilde{E}$ such that the projections of $\zeta$ and $\xi$ in $E$ are disjoint. Define the {\em equivariant linking number} 
of $\zeta$ and $\xi$ as: 
$$\lk_e(\zeta,\xi)=\frac{1}{P(t)}\sum_{j\in\Z}\langle\Sigma,\tau^j(\xi)\rangle\,t^j\ \in\Qt.$$
The equivariant linking number is well-defined since $H_2(\tilde{E},\Q)=0$ (see for instance \cite[Lemma 2.1]{M4}) 
and satisfies $\lk_e(\tau(\zeta),\xi)=t\,\lk_e(\zeta,\xi)$. 

First consider a tangle with disks $(\gamma,k)$, with disks $d_\ell$ associated to the integer $k$, defined from a surgery presentation of a bottom-top tangle with paths in a $\Q$--cube, so that the closure 
$\bar{\gamma}_i$ of each component $\gamma_i$ is well-defined. Fix a diagram of $(\gamma,k)$ and base points $\star_i$ of 
its components. For an interval component, choose the base point to be its origin. Set $d=\cup_{\ell=1}^kd_\ell$. Let $E$ be the exterior of 
$\partial d$ in $[-1,1]^3$ and let $\tilde{E}$ be the infinite cyclic covering defined above. Let $\tilde{E}_0\subset\tilde{E}$ be a copy of the exterior of $d$ 
in $[-1,1]^3$. Define the lift $\tilde{\gamma_i}$ of $\bar{\gamma}_i$ in $\tilde{E}$ by lifting $\star_i$ in $\tilde{E}_0$. Given two subtangles $\gamma_I$ and 
$\gamma_J$ of $\gamma$, define the {\em equivariant linking matrix} $\Lk_e(\tilde{\gamma}_I,\tilde{\gamma}_J)$ of their lifts with the equivariant linking numbers 
of the $\tilde{\gamma_i}$. 
\begin{lemma}
 $W_{\gamma_I\gamma_J}=\Lk_e(\tilde{\gamma}_I,\tilde{\gamma}_J)$
\end{lemma}
\begin{proof}
 Set $\tilde{E}_\ell=\tau^\ell(\tilde{E}_0)$, where $\tau$ is the generator of the automorphism group of the covering $\tilde{E}$ which corresponds to the action of the 
 positive meridians of the $\partial d_\ell$.  
 Fix $i\in I$ and $j\in J$. Since $\bar{\gamma}_i$ is null-homotopic in $[-1,1]^3\setminus\partial d$, it bounds a disk $D$ immersed in $[-1,1]^3\setminus\partial d$. 
 Let $\tilde{D}$ be the lift of $D$ obtained by lifting $\star_i$ in $\tilde{E}_0$. Set $\tilde{D}_\ell=\tilde{D}\cap\tilde{E}_\ell$ and let $D_\ell$ be the 
 image of $\tilde{D}_\ell$ in $E$. Set $c_\ell=\partial D_\ell$ and $\tilde{c}_\ell=\partial\tilde{D}_\ell$. Similarly, define the $c'_\ell$ and $\tilde{c}'_\ell$ from 
 $\bar{\gamma}_j$. Assume the $c'_\ell$ do not meet the $D_\ell$ along the disks of the tangle. 
 
 Thanks to Lemma \ref{lemmawdmatrix1}, we have $w(\gamma_i,\gamma_j)=\sum_{\ell,\ell'\in\Z}w(c_\ell,c'_{\ell'})t^{\ell-\ell'}$ 
 for any choice of base points of the $c_\ell$ and $c'_{\ell'}$. Since these latter curves do not cross the disks of the tangle, we have  
 $w(c_\ell,c'_{\ell'})=\lk(c_\ell,c'_{\ell'})=\langle D_\ell,c'_{\ell'}\rangle$, thus 
 $w(\gamma_i,\gamma_j)=\sum_{\ell,\ell'\in\Z}\langle D_\ell,c'_{\ell'}\rangle t^{\ell-\ell'}$. Lifting both $D_\ell$ and $c'_{\ell'}$ in $\tilde{E}_\ell$ 
 does not change their algebraic intersection number, hence
 \begin{eqnarray*}
  w(\gamma_i,\gamma_j) & = & \sum_{\ell,\ell'\in\Z}\langle \tilde{D}_\ell,\tau^{\ell-\ell'}(\tilde{c}'_{\ell'})\rangle t^{\ell-\ell'} \\
  & = & \sum_{\ell,\ell'\in\Z}\langle \tilde{D}_\ell,\tau^{\ell'}(\tilde{c}'_{\ell-\ell'})\rangle t^{\ell'} \\
  & = & \sum_{\ell'\in\Z}\langle \tilde{D},\tau^{\ell'}(\tilde{\gamma_j})\rangle t^{\ell'} \\
  & = & \lk_e(\tilde{\gamma}_i,\tilde{\gamma_j})
 \end{eqnarray*}
 where the third equality holds since $\tau^{\ell'}(\tilde{c}'_{\ell-\ell'})=\tau^{\ell'}(\tilde{\gamma_j})\cap\tilde{E}_\ell$. 
\end{proof}

Now consider a bottom-top tangle with paths in a $\Q$--cube $(B,K,\gamma)$. Since $\gamma$ is null-homotopic in $B\setminus K$, we have a well-defined 
{\em equivariant linking matrix} $\Lk_e(\tilde{\gamma})$. Here, all components are intervals, so we have a canonical choice of base points. 
\begin{proposition} \label{propintertopo}
 Let $(B,K,\gamma)$ be a bottom-top tangle with paths in a $\Q$--cube. Then $W_\gamma=\Lk_e(\tilde{\gamma})$. 
\end{proposition}
\begin{proof}
 Let $(([-1,1]^3,\Xi,\eta),L)$ be a surgery presentation of $(B,K,\gamma)$. Fix a diagram of $([-1,1]^3,\Xi,\eta\cup L)$ and base points for the components of 
 $L=\sqcup_{1\leq i\leq n}L_i$. Let $\tilde{E}$ be the infinite cyclic covering of the exterior of $\Xi$ in $[-1,1]^3$. Let $d$ be the disjoint union of disks in $[-1,1]^3$ 
 bounded by $\bXi$. Let $\tilde{L}=\sqcup_{1\leq i\leq n}\tilde{L}_i$, $\tilde{\gamma}$ and $\tilde{\eta}$ be the lifts of $L$, $\gamma$ and $\eta$ in $\tilde{E}$ 
 with all base points in the same copy in $\tilde{E}$ of the exterior of $d$ in $[-1,1]^3$. 
 We have to prove that:
 $$\Lk_e(\tilde{\gamma})=\Lk_e(\tilde{\eta})-\Lk_e(\tilde{\eta},\tilde{L})\Lk_e(\tilde{L})^{-1}\Lk_e(\tilde{L},\tilde{\eta}).$$
 The infinite cyclic covering $\tilde{E}'$ of the exterior of $K$ in $B$ is obtained from $\tilde{E}$ by surgery on $\cup_{\ell\in\Z}\tau^\ell(\tilde{L})$. 
 Let $\tilde{\eta}_x$ and $\tilde{\eta}_y$ be components of $\tilde{\eta}$, and let $\tilde{\gamma}_x$ and $\tilde{\gamma}_y$ be the corresponding components of $\tilde{\gamma}$. 
 For any knot $\lambda$ in $\tilde{E}$ or $\tilde{E}'$, denote $m(\lambda)$ a positive meridian. For $1\leq i\leq n$, let $c_i$ be the parallel of $\tilde{L}_i$ 
 which bounds a disk after surgery. In the group $H_1\left(\tilde{E}\setminus\cup_{\ell\in\Z}\left(\tau^\ell(\tilde{L})\cup\tau^\ell(\tilde{\eta}_y)\right);\Z\right)$, 
 we have
 $$\tilde{\eta}_x=\lk_e(\tilde{\eta}_x,\tilde{\eta}_y)m(\tilde{\eta}_y)+\sum_{i=1}^n\lk_e(\tilde{\eta}_x,\tilde{L}_i)m(\tilde{L}_i)$$
 and
 $$c_i=\lk_e(\tilde{L}_i,\tilde{\eta}_y)m(\tilde{\eta}_y)+\sum_{j=1}^n\lk_e(\tilde{L}_i,\tilde{L}_j)m(\tilde{L}_j),$$
 where multiplication by $t$ is given by the action of $\tau$ in homology. 
 In $H_1(\tilde{E}'\setminus\cup_{\ell\in\Z}\tau^\ell(\tilde{\gamma}_y);\Z)$, this gives 
 $\tilde{\gamma}_x=\lk_e(\tilde{\eta}_x,\tilde{\eta}_y)m(\tilde{\gamma}_y)-\Lk_e(\tilde{\eta}_x,\tilde{L})\Lk_e(\tilde{L})^{-1}\Lk_e(\tilde{L},\tilde{\eta}_y)m(\tilde{\gamma}_y).$
\end{proof}

It is easily checked that the null LP--surgery defined in the introduction provides a move on the set of bottom-top tangles with paths in $\Q$--cubes, which corresponds to the move of null LP--surgery on Lagrangian cobordisms defined at the end of Subsection~\ref{subseccob} {\em via} the map $D$ of Proposition~\ref{propmapD}. Moreover:
\begin{corollary} \label{propinvariancewind}
 The winding matrix of a bottom-top tangle with paths in a $\Q$--cube is invariant under null LP--surgeries.
\end{corollary}
\begin{proof}
 Let $(B,K,\gamma)$ be a bottom-top tangle with paths in a $\Q$--cube associated with a Lagrangian cobordism with paths $(M,K)$. Let $\chir$ be a null LP--surgery 
 on $(M,K)$ made of a single $\Q$SK--pair. Let $(M',K')$ be the Lagrangian cobordism with paths obtained by surgery and let $(B',K',\gamma')$ be the associated 
 bottom-top tangle with paths in a $\Q$--cube. Let $\tilde{E}$ be the infinite cyclic covering of the exterior of $K$ in $B$. The nullity condition 
 implies that the preimage of the $\Q$--handlebody $C$ is the disjoint union of $\Q$--handlebodies $C_\ell$ homeomorphic to $C$. The infinite cyclic covering 
 $\tilde{E}'$ of the exterior of $K'$ in $B'$ is obtained from $\tilde{E}$ by null LP--surgeries on all the $C_\ell$. This concludes since LP--surgeries preserve 
 linking numbers (see for instance \cite[Lemma 2.1]{M3}).
\end{proof}

      \section{Construction of an LMO functor on $\tlcob$} \label{secfunctor}

    \subsection{The functor $\Zp:\ttang\to\tAw_{\Qtt}$}

The definition of the functor $\Zp:\ttang\to\tAw_{\Qtt}$ is based on the functor $Z:\tang\to\A$ of \cite{CHM}, which is a renormalization of the Le--Murakami functor 
\cite{LM1,LM2}. We recall in Figure \ref{figfunctorZ} the definition of $Z$ on the elementary $q$--tangles, where $\nu\in\A(\xc)\cong\A(\xd)$ is the value of the 
Kontsevich integral on the zero framed unknot, $\Phi\in\A(\xdop\xdop\xdop)$ is a Drinfeld associator with rational coefficients and 
$\Delta^{+++}_{u_1,u_2,u_3}:\A(\xdop\xdop\xdop)\to\A(\xdop_{u_1u_2u_3})$ is obtained by applying $(|u_i|-1)$ times $\Delta$ on the $i$-th factor.
\begin{figure}[htb] 
$$Z\left(\hspace{-1ex}\raisebox{-0.8cm}{
\begin{tikzpicture} [scale=0.6]
 \draw[->] (0,1) node[above] {$\scriptstyle{(+}$} -- (1,0) node[below] {$\scriptstyle{+)}$};
 \draww{[->] (1,1) node[above] {$\scriptstyle{+)}$} -- (0,0) node[below] {$\scriptstyle{(+}$};}
\end{tikzpicture}}
\right)=\exp\left(\frac{1}{2}\raisebox{-0.5cm}{
\begin{tikzpicture} [scale=0.6]
 \draw[->] (0,2) -- (0,1) -- (1,0);
 \draw[->] (1,2) -- (1,1) -- (0,0);
 \draw[dashed] (0,1.5) -- (1,1.5);
\end{tikzpicture}}\,\right)\in\A\left(\hspace{-0.7ex}\raisebox{-0.2cm}{
\begin{tikzpicture} [scale=0.6]
 \draw[->] (0,1) -- (1,0);
 \draw[->] (1,1) -- (0,0);
\end{tikzpicture}}\right)
\qquad
Z\left(\hspace{-1ex}\raisebox{-0.8cm}{
\begin{tikzpicture} [scale=0.6]
 \draw[->] (1,1) node[above] {$\scriptstyle{+)}$} -- (0,0) node[below] {$\scriptstyle{(+}$};
 \draww{[->] (0,1) node[above] {$\scriptstyle{(+}$} -- (1,0) node[below] {$\scriptstyle{+)}$};}
\end{tikzpicture}}
\right)=\exp\left(-\frac{1}{2}\raisebox{-0.5cm}{
\begin{tikzpicture} [scale=0.6]
 \draw[->] (0,2) -- (0,1) -- (1,0);
 \draw[->] (1,2) -- (1,1) -- (0,0);
 \draw[dashed] (0,1.5) -- (1,1.5);
\end{tikzpicture}}\,\right)\in\A\left(\hspace{-0.7ex}\raisebox{-0.2cm}{
\begin{tikzpicture} [scale=0.6]
 \draw[->] (0,1) -- (1,0);
 \draw[->] (1,1) -- (0,0);
\end{tikzpicture}}\right)$$
$$Z\left(\hspace{-1ex}\raisebox{-0.4cm}{
\begin{tikzpicture} [scale=0.6]
 \draw[->] (1,0) node[below] {$\scriptstyle{-)}$} .. controls +(0,1) and +(0,1) .. (0,0) node[below] {$\scriptstyle{(+}$};
\end{tikzpicture}}
\right)=\raisebox{-0.1cm}{
\begin{tikzpicture} [scale=0.6]
 \draw[->] (2,0) .. controls +(0,0.5) and +(0.5,0) .. (1.5,1) -- (0.5,1) .. controls +(-0.5,0) and +(0,0.5) .. (0,0);
 \draw[fill=white] (1,1) circle (0.3);
 \draw (1,1) node {$\nu$};
\end{tikzpicture}}\,\in\A\left(\hspace{-0.7ex}\raisebox{-0.05cm}{
\begin{tikzpicture} [scale=0.6]
 \draw[->] (1,0) .. controls +(0,1) and +(0,1) .. (0,0);
\end{tikzpicture}}\right)
\qquad
Z\left(\hspace{-1ex}\raisebox{-0.5cm}{
\begin{tikzpicture} [scale=0.6]
 \draw[->] (0,0) node[above] {$\scriptstyle{(+}$} .. controls +(0,-1) and +(0,-1) .. (1,0) node[above] {$\scriptstyle{-)}$};
\end{tikzpicture}}
\right)=\raisebox{-0.3cm}{
\begin{tikzpicture} [scale=0.7]
 \draw[->] (0,0) .. controls +(0,-1) and +(0,-1) .. (1,0);
\end{tikzpicture}}\,\in\A\left(\hspace{-0.7ex}\raisebox{-0.25cm}{
\begin{tikzpicture} [scale=0.6]
 \draw[->] (0,0) .. controls +(0,-1) and +(0,-1) .. (1,0);
\end{tikzpicture}}\right)$$
$$Z\left(\hspace{-1ex}\raisebox{-0.7cm}{
\begin{tikzpicture} [scale=0.6]
 \draw[->] (0,1) node[above] {$\scriptstyle{(u}$} -- (0,0) node[below] {$\scriptstyle{((u}\ $};
 \draw[->] (1.5,1) node[above] {$\scriptstyle{(v}\,$} -- (0.5,0) node[below] {$\,\scriptstyle{v)}$};
 \draw[->] (2,1) node[above] {$\ \scriptstyle{w))}$} -- (2,0) node[below] {$\scriptstyle{w)}$};
\end{tikzpicture}}\right)=\Delta^{+++}_{u,v,w}(\Phi)\,\in\A(\xdop_{uvw})$$
\caption{The functor $Z:\tang\to\A$.} \label{figfunctorZ}
\end{figure}

Let $(\gamma,k)$ be a $q$--tangle with disks. Assume $\gamma$ is transverse to $[-1,1]^2\times\{h_i(k)\}$ 
for all $i\in\{1,\dots,k\}$, and write $\gamma$ as a composition of $q$--tangles $\gamma_i$ by cutting along these levels, see Figure \ref{figdecouptangle}. 
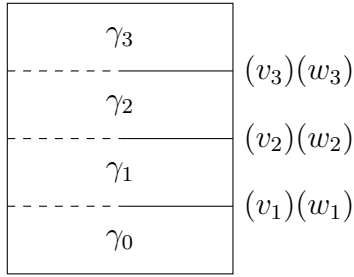
\begin{figure}[htb] 
\begin{center}
\begin{tikzpicture} [xscale=1.5,yscale=0.9]
 \draw (0,0) -- (2,0) -- (2,4) -- (0,4) -- (0,0);
 \foreach \x in {1,2,3} {\draw[dashed] (0,\x) -- (1,\x); \draw (1,\x) -- (2,\x); \draw (2,\x) node[right] {$(v_\x)(w_\x)$};}
 \foreach \x in {0,1,2,3} {\draw (1,\x+0.5) node {$\gamma_\x$};}
\end{tikzpicture}
\end{center} \caption{Cutting a $q$--tangle with disks $(\gamma,3)$.} \label{figdecouptangle}
\end{figure}
Write the bottom word of $\gamma_i$ as $w_b(\gamma_i)=(v_i)(w_i)$, where $w_i$ corresponds to the part of the tangle which meets the disk $d_i$. Set:
$$\Zp(\gamma,k)=Z(\gamma_0)\circ(I_{v_1}\otimes G_{w_1})\circ Z(\gamma_1)\circ\dots\circ(I_{v_k}\otimes G_{w_k})\circ Z(\gamma_k)\ \in\tAw_{\Qtt}(\gamma),$$
where $I_v$ is the identity on the word $v$ and $G_v$ is obtained from $I_v$ by adding a label $t$ (resp. $t^{-1}$) on skeleton components associated with 
a $-$ sign (resp. a $+$ sign), see Figure~\ref{figIvGv}. 
\begin{figure}[htb] 
\begin{center}
\begin{tikzpicture} [xscale=0.6,yscale=0.5]
 \draw (0,1) node {$I_{\scriptscriptstyle--+-}=$};
 \draw[->] (2,0) -- (2,2); \draw[->] (3,0) -- (3,2); \draw[<-] (4,0) -- (4,2); \draw[->] (5,0) -- (5,2);
\begin{scope} [xshift=10cm]
 \draw (0,1) node {$G_{\scriptscriptstyle--+-}=$};
 \draw[->] (2,0) -- (2,2); \draw[->] (3,0) -- (3,2); \draw[<-] (4,0) -- (4,2); \draw[->] (5,0) -- (5,2);
 \foreach \x in {2,3,4,5} {\draw (\x,1) node {$\scriptscriptstyle{\bullet}$};}
 \foreach \x in {2,3,5} {\draw (\x,1) node[right] {$\scriptstyle{t}$};}
 \draw (3.9,1.1) node[right] {$\scriptstyle{t^{-1}}$};
\end{scope}
\end{tikzpicture}
\end{center} \caption{The diagrams $I_v$ and $G_v$.} \label{figIvGv}
\end{figure}
At the level of objects, $\Zp$ forgets the parentheses. 
Invariance with respect to isotopy and to the cutting of $\gamma$ is due to invariance of the functor $Z$ and the following observation of Kricker 
\cite[Lemma 3.2.4]{Kri}.
\begin{lemma}
 For a winding Jacobi diagram $D\in\tAw_{\Qtt}(w,v)$, we have $G_v\circ D=D\circ G_w$.
\end{lemma}
\begin{proof}
 Apply the relations Hol and \holw\ at all vertices of the diagram.
\end{proof}
Furthermore, $\Zp$ is a clearly a functor and it preserves the tensor product on $\tang\otimes\ttang$ since $Z$ is tensor-preserving. 
\begin{lemma} \label{lemmagrouplike}
 For any $q$--tangle with disks $(\gamma,k)$, $\Zp(\gamma,k)$ is group-like.
\end{lemma}
\begin{proof}
 The fact that $Z(\gamma)$ is group-like for a $q$--tangle $\gamma$ follows from \cite[Theorem~5.1]{LM3}. This concludes since the $G_v$ are obviously group-like and the coproduct commutes with the composition. 
\end{proof}

    \subsection{The functor $Z:\ttcub\to\tAw_{\Qt}$}

The next step is to evaluate $\Zp$ on the surgery presentation of a $q$--tangle with paths in a $\Q$--cube. Let $(B,K,\gamma)\in\ttcub(w,v)$. Let 
$(([-1,1]^3,\Xi,\eta),L)$ be a surgery presentation of $(B,K,\gamma)$. 
The trivial link $\hat{\Xi}$ is the union of the boundaries of the disks $d_i=[0,1]\times[-1,1]\times\{h_i(k)\}$, where $k$ is the number of components of $\Xi$. 
Hence we have a $q$--tangle with disks $(\eta\cup L,k)$ and $\Zp(\eta\cup L,k)\in\tAw_{\Qtt}(\eta\cup L)$. Set:
$$\Zc((\Xi,\eta),L)=\chi_{\pi_0(L)}^{-1}(\nu^{\otimes\pi_0(L)}\sharp_{\pi_0(L)}\Zp(\eta\cup L,k))\ \in\tAw_{\Qtt}(\eta,\xlw_{\pi_0(L)})$$
where the connected sum means that a copy of $\nu$ is summed to each component of $L$. Note that $\Zc((\Xi,\eta),L)$ is group-like since $\Zp(\eta\cup L,k)$ 
and $\nu$ are group-like and $\chi_{\pi_0(L)}$ preserves the coproduct.

We want to apply formal Gaussian integration to $\Zc((\Xi,\eta),L)$. We work with a lift $\overline{\Zc((\Xi,\eta),L)}\in\tAw_{\Qtt}(\eta,*_{\pi_0(L)})$. 
Fix a diagram of the $q$--tangle with disks $(\eta\cup L,k)$ transverse to the levels $\{h_i(k)\}$, and fix base points $\star_i$ on each component $L_i$ 
of $L$. Construct $\overline{\Zc((\Xi,\eta),L)}$ following the construction from the beginning of Section \ref{secfunctor} for this diagram, with the skeleton components 
corresponding to the components of $L$ defined as intervals by cutting each component $L_i$ at the base point $\star_i$. 

\begin{lemma} \label{lemmawdmatrix}
 The lift $\overline{\Zc((\Xi,\eta),L)}$ is group-like, and we have:
 $$\overline{\Zc((\Xi,\eta),L)}=\expd\left(\frac{1}{2}W_L\right)\sqcup H,$$ 
 where $W_L$ is the winding matrix associated with our choice of diagram and base points and $H$ is $\pi_0(L)$--substantial. 
\end{lemma}
\begin{proof}
 Check as in Lemma \ref{lemmagrouplike} that $\overline{\Zc((\Xi,\eta),L)}$ is group-like. We have to compute the part of $\overline{\Zc((\Xi,\eta),L)}$ 
 made of $\pi_0(L)$--struts. We work with $\chi_{\pi_0(\eta)}^{-1}(\overline{\Zc((\Xi,\eta),L)})\in\tA_{\Qtt}(*_{\pi_0(\eta)\cup\pi_0(L)})$, which is also group-like, in order to have a Hopf algebra 
 structure on our diagram space. In particular, the group-like property implies that $\chi_{\pi_0(\eta)}^{-1}(\overline{\Zc((\Xi,\eta),L)})$ is the exponential 
 of a series of connected diagrams. Since $\nu$ and the associator $\Phi$ have no terms with exactly two vertices, the only contributions to the $\pi_0(L)$--struts part 
 come from the crossings between components of $L$. For $i\neq j$, the definition of $Z$ and the \holw\ relation show that the contribution 
 of a crossing $c$ between $L_i$ and $L_j$ is $\chi_{\pi_0(\eta)}^{-1}\raisebox{-0.8ex}{{\textrm{\Huge(}}}\frac{1}{2}\textrm{sg}(c)
 \raisebox{-1cm}{
 \begin{tikzpicture} [xscale=0.6,yscale=0.5]
 \draw[->] (0,0) -- (0,2); \draw (0,0) node[below] {$\scriptstyle{L_i}$}; \draw[->] (2,0) -- (2,2); \draw (2,0) node[below] {$\scriptstyle{L_j}$};
 \draw[dashed,->] (0,1) -- (1,1); \draw[dashed] (1,1) -- (2,1); \draw (1,1) node[above] {$\scriptstyle{t^{\varepsilon_{ij}(c)}}$};
 \end{tikzpicture}}$\raisebox{-0.8ex}{{\Huge)}}. 
 Hence the contribution of all crossings between $L_i$ and $L_j$ is \raisebox{-1cm}{
 \begin{tikzpicture} [xscale=0.7,yscale=0.5]
 \begin{scope}
  \draw[dashed,->] (0,0) -- (0,1); \draw[dashed] (0,1) -- (0,2);
  \draw (0,0) node[below] {$\scriptstyle{L_i}$}; \draw (0,2) node[above] {$\scriptstyle{L_j}$}; \draw (0,1) node[right] {$\scriptstyle{(W_L)_{ij}}$};
 \end{scope}
  \draw (2.3,0.7) node {$=$};
 \begin{scope} [xshift=3.3cm]
  \draw[dashed,->] (0,0) -- (0,1); \draw[dashed] (0,1) -- (0,2);
  \draw (0,0) node[below] {$\scriptstyle{L_j}$}; \draw (0,2) node[above] {$\scriptstyle{L_i}$}; \draw (0,1) node[right] {$\scriptstyle{(W_L)_{ji}}$};
 \end{scope}
 \end{tikzpicture}}.
 For $i=j$, the contribution of a self-crossing of $L_i$ is:
 $$\chi_{\pi_0(\eta)}^{-1}\left(\frac{1}{2}\textrm{sg}(c)
 \raisebox{-0.9cm}{
 \begin{tikzpicture} [xscale=0.7,yscale=0.7]
  \draw[->] (0,0) -- (0,2); \draw (0,0) node[below] {$\scriptstyle{L_i}$};
  \draw[dashed,->] (0,0.5) arc (-90:0:0.5); \draw[dashed] (0.5,1) arc (0:90:0.5); \draw (0.5,1) node[right] {$\scriptstyle{t^{\varepsilon_{ii}(c)}}$};
 \end{tikzpicture}}\right)
 =\textrm{sg}(c)\left(\raisebox{-0.9cm}{
 \begin{tikzpicture} [xscale=0.7,yscale=0.5]
  \draw[dashed,->] (0,0) -- (0,1); \draw[dashed] (0,1) -- (0,2);
  \draw (0,0) node[below] {$\scriptstyle{L_i}$}; \draw (0,2) node[above] {$\scriptstyle{L_i}$}; \draw (0,1) node[right] {$\scriptstyle{t^{\varepsilon_{ii}(c)}}$};
 \end{tikzpicture}}
 +\frac{1}{2}\raisebox{-0.9cm}{
 \begin{tikzpicture} [xscale=0.5,yscale=0.5]
  \draw[dashed] (0,0) -- (0,1); \draw[dashed,->] (0,1) arc (-90:90:0.5); \draw[dashed] (0,2) arc (90:270:0.5);
  \draw (0,0) node[below] {$\scriptstyle{L_i}$}; \draw (0,2) node[above] {$\scriptstyle{t^{\varepsilon_{ii}(c)}}$};
 \end{tikzpicture}}\right).
 $$ 
 Summed over all self-crossings of $L_i$, we get as strut part:
 $$\sum_c\frac{1}{2}\textrm{sg}(c)\raisebox{-0.9cm}{
 \begin{tikzpicture} [xscale=0.7,yscale=0.5]
  \draw[dashed,->] (0,0) -- (0,1); \draw[dashed] (0,1) -- (0,2);
  \draw (0,0) node[below] {$\scriptstyle{L_i}$}; \draw (0,2) node[above] {$\scriptstyle{L_i}$}; \draw (0,1) node[right] {$\scriptstyle{t^{\varepsilon_{ii}(c)}}$};
 \end{tikzpicture}}
 =\frac{1}{2}\raisebox{-0.9cm}{
 \begin{tikzpicture} [xscale=0.7,yscale=0.5]
  \draw[dashed,->] (0,0) -- (0,1); \draw[dashed] (0,1) -- (0,2);
  \draw (0,0) node[below] {$\scriptstyle{L_i}$}; \draw (0,2) node[above] {$\scriptstyle{L_i}$}; \draw (0,1) node[right] {$\scriptstyle{(W_L)_{ii}}$};
 \end{tikzpicture}}.$$
 Hence $\chi_{\pi_0(\eta)}^{-1}(\overline{\Zc((\Xi,\eta),L)})=\expd\left(\frac{1}{2}W_L\right)\sqcup H'$ where $H'\in\tA_{\Qtt}(*_{\pi_0(\eta)\cup\pi_0(L)})$ 
 is $\pi_0(L)$--substantial. Set $H=\chi_{\pi_0(\eta)}(H')$.
\end{proof}
The matrix $W_L(1)$ is the linking matrix of the link $L$, hence it is the presentation matrix of the first homology group of a $\Q$--cube. 
Thus $\det(W_L(1))\neq0$ and $\overline{\Zc((\Xi,\eta),L)}$ is a non-degenerate Gaussian. Lemma \ref{lemmaFGI} implies:
\begin{lemma}
 The formal Gaussian integral $\int_{\pi_0(L)}\overline{\Zc((\Xi,\eta),L)}$ does not depend on the lift 
 $\overline{\Zc((\Xi,\eta),L)}\in\tAw_{\Qtt}(\eta,*_{\pi_0(L)})$ of $\Zc((\Xi,\eta),L)\in\tAw_{\Qtt}(\eta,\xlw_{\pi_0(L)})$.
\end{lemma}
This allows to set:
$$\int_{\pi_0(L)}\Zc((\Xi,\eta),L)=\int_{\pi_0(L)}\overline{\Zc((\Xi,\eta),L)}\quad\in\tAw_{\Qt}(\gamma).$$

\begin{proposition}
 Let $(B,K,\gamma)$ be a $q$--tangle with paths in a $\Q$--cube. Fix a surgery presentation $(([-1,1]^3,\Xi,\eta),L)$ of $(B,K,\gamma)$. Then:
 $$Z(B,K,\gamma)=U_+^{-\sigma_+(L)}\sqcup U_-^{-\sigma_-(L)}\sqcup \int_{\pi_0(L)}\Zc((\Xi,\eta),L)\quad\in\tAw_{\Qt}(\gamma),$$
 where $U_{\pm}=\Zc((\varnothing,\varnothing),\raisebox{-0.1cm}{
 \begin{tikzpicture}
  \draw[->] (0,0) arc (0:360:0.2) node[right] {$\scriptstyle{\pm1}$};
 \end{tikzpicture}})$, defines a functor $Z:\ttcub\to\tAw_{\Qt}$ which preserves the tensor product on $\tcub\otimes\ttcub$. 
\end{proposition}
\begin{proof}
 We have to check that $Z(B,K,\gamma)$ does not depend on the surgery presentation. Independance with respect to the orientation of the components of $L$ 
 follows from the argument of \cite[Proposition 3.1]{AA2}. The normalization term $U_+^{-\sigma_+(L)}\sqcup U_-^{-\sigma_-(L)}$ ensures independance with respect 
 to the KI move as usual. Independance with respect to the KII move mainly follows from \cite[Section 5.4]{GK}. More precisely, the argument of \cite[Theorem 4]{GK} 
 adapts \cite[Proposition 1]{LMMO} to relate the values of $\overline{\Zc((\Xi,\eta),L)}$ for surgery links that differ from each other by 
 a KII move. Then \cite[Lemma 5.6]{GK} shows that this implies the invariance of the formal Gaussian integral. As noted in \cite[Section 5.1]{AA2}, 
 the argument remains valid when a surgery component is added to a tangle component since \cite[Lemma 5.6]{GK} uses integration along the surgery component. 
\end{proof}

Restricting the functor $Z:\ttcub\to\tAw_{\Qt}$ to $q$--tangles in $\Q$--cubes with no path, one recovers the functor $Z:\tcub\to\A$ of \cite[Definition 3.16]{CHM}. 
When $\gamma$ is a bottom-top tangle and $K=\varnothing$, $\chi_{\pi_0(\gamma)}^{-1}(Z(B,\varnothing,\gamma))$ is group-like and 
$\chi_{\pi_0(\gamma)}^{-1}(Z(B,\varnothing,\gamma))=\expd(\Lk(\gamma))\sqcup H$ for some substantial and group-like $H$ \cite[Lemma 3.17]{CHM}. 
We generalize this in the next result. 
\begin{lemma} \label{lemmagrouplikeZ}
 For any bottom-top $q$--tangle with paths $(B,K,\gamma)$ where $B$ is a $\Q$--cube, $\chi_{\pi_0(\gamma)}^{-1}(Z(B,K,\gamma))$ is group-like
 and:
 $$\chi_{\pi_0(\gamma)}^{-1}(Z(B,K,\gamma))=\expd(W_\gamma)\sqcup H\quad\in\tA_{\Qt}(*_{\pi_0(\gamma)}),$$ for some substantial and group-like $H$. 
\end{lemma}
\begin{proof}
 The fact that $Z(B,K,\gamma)$ is group-like follows from the same property for $U_+$, $U_-$ and $\overline{\Zc((\Xi,\eta),L)}$, and Theorem \ref{thJMM}. 
 It implies that $\chi_{\pi_0(\gamma)}^{-1}(Z(B,K,\gamma))$ is group-like since $\chi_{\pi_0(\gamma)}^{\phantom{-1}}$ preserves the coproduct.
 The same computation as in the proof of Lemma \ref{lemmawdmatrix} gives:
 $$\chi_{\pi_0(\eta)}^{-1}\left(\overline{\Zc((\Xi,\eta),L)}\right)=\expd\left(\frac{1}{2}W_L+\frac{1}{2}W_{\eta}+W_{L\eta}\right)\sqcup H',$$
 where $H'$ is substantial ---note that $\chi_{\pi_0(\gamma)}$ and $\chi_{\pi_0(\eta)}$ are essentially the same before and after surgery on $L$. Integrate along $\pi_0(L)$:
 \begin{eqnarray*}
 \chi_{\pi_0(\gamma)}^{-1}\left(\int_{\pi_0(L)}\overline{\Zc((\Xi,\eta),L)}\right)&&\\
 &\hspace{-6cm}=&\hspace{-3cm}
 \left\langle\expd\left(-\frac{1}{2}W_L^{-1}\right),\expd\left(\frac{1}{2}W_{\eta}+W_{L\eta}\right)\sqcup H'\right\rangle_{\pi_0(L)}\\
 &\hspace{-6cm}=&\hspace{-3cm}\expd\left(\frac{1}{2}W_{\eta}-\frac{1}{2}{}^tW_{L\eta}(t^{-1})W_L^{-1}W_{L\eta}\right)\sqcup H\\
 &\hspace{-6cm}=&\hspace{-3cm}\expd\left(\frac{1}{2}W_{\gamma}\right)\sqcup H.
 \end{eqnarray*}
\end{proof}

    \subsection{The functor $\tZ:\tlcob_q\to\tAts$}

In this section, we define a functor on Lagrangian $q$--cobordisms with paths by applying the invariant $Z$ on bottom-top $q$--tangles with paths in $\Q$--cubes. 
The invariant $Z$ is functorial on $q$--tangles but not on bottom-top $q$--tangles, due to the different composition laws. To deal with this, we introduce 
some specific elements $\T_g\in\tAtsm{g}{g}$ following \cite[Sec. 4]{CHM}. 
Set:
$$\lambda(x,y;r)=\chi_{\{r\}}^{-1}\left(\left(\sum_{n\geq0}\frac 1{n!}\raisebox{-6ex}{
\begin{tikzpicture} [scale=0.35]
 \draw[->] (0,0) node[below] {$\scriptstyle{r}$} -- (0,5);
 \foreach \y in {1,3,4} {
 \draw[dashed] (0,\y) -- (2,\y) node[right] {$\scriptstyle{x}$};}
 \draw (1,2) node {$\vdots$};
 \draw (3.8,2.5) node {$\left.\resizebox{0cm}{0.75cm}{\phantom{rien}}\right\rbrace \scriptstyle{n}$};
\end{tikzpicture}}\right)\circ\left(\sum_{n\geq0}\frac 1{n!}\raisebox{-6ex}{
\begin{tikzpicture} [scale=0.35]
 \draw[->] (0,0) node[below] {$\scriptstyle{r}$} -- (0,5);
 \foreach \y in {1,3,4} {
 \draw[dashed] (0,\y) -- (2,\y) node[right] {$\scriptstyle{y}$};}
 \draw (1,2) node {$\vdots$};
 \draw (3.8,2.5) node {$\left.\resizebox{0cm}{0.75cm}{\phantom{rien}}\right\rbrace \scriptstyle{n}$};
\end{tikzpicture}}\right)\right)
\in\tA_{\Qt}(*_{\{x,y,r\}}),$$
$$\T(x^+,x^-)=U_+^{-1}\sqcup U_-^{-1}\sqcup\int_{\{r^+,r^-\}}\langle\lambda(1^-,x^-;r^-)\sqcup\lambda(x^+,1^+;r^+),\chi^{-1}(Z(T_1))\rangle_{\{1^+,1^-\}},$$
$$\T_g=\T(1^+,1^-)\sqcup\dots\sqcup\T(g^+,g^-)\in\tA_{\Qt}(\pgf{g}{g}),$$
where the bottom-top tangle $T_1$ is drawn in Figure \ref{figTg}.
As proven in \cite[Lemma 4.9]{CHM}:
\begin{lemma} \label{lemmagrouplikeTg}
 $\T_g$ is a group-like element of $\tA_{\Qt}(\pgf{g}{g})$ and $\T_g=Id_g\sqcup H$ for some substantial and group-like $H$. In particular, $\T_g$ is top-substantial 
 and $\lfloor g\rceil^-$--substantial.
\end{lemma}

Let $(M,K)$ be a Lagrangian $q$--cobordism with paths and denote $(B,K,\gamma)$ the associated bottom-top $q$--tangle with paths, of type $(g,f)$. 
We have $Z(B,K,\gamma)\in\tAw_{\Qt}(\gamma)\cong\tA_{\Qt}(\gamma)$ and we consider $\chi^{-1}(Z(B,K,\gamma))\in\tA_{\Qt}(\pgf{g}{f})$. 
It may not be top-substan\-tial, but since $\T_g$ is $\lfloor g\rceil^-$--substantial, we can set:
$$\tZ(M,K)=\chi^{-1}(Z(B,K,\gamma))\circ\T_g.$$
At the level of objects, $\tZ$ sends a word on its number of letters. 
Direct adaptation of the proof of \cite[Lemma 4.10]{CHM} implies that $\tZ$ preserves the composition and the next result follows, see \cite[Theorem 4.13]{CHM}. 

\begin{proposition}
 $\tZ:\tlcob_q\to\tAts$ is a functor which preserves the tensor product on $\lcob_q\otimes\tlcob_q$.
\end{proposition}
Restricting the functor $\tZ:\tlcob_q\to\tAts$ to Lagrangian $q$--cobordisms with no path, one recovers the functor $\tZ$ defined on $\lcob_q$ in \cite[Theorem 4.13]{CHM}. 

Lemmas \ref{lemmagrouplikeZ} and \ref{lemmagrouplikeTg} imply:
\begin{lemma}
 Let $(M,K)$ be a Lagrangian $q$--cobordism with paths and let $(B,K,\gamma)$ be the associated bottom-top $q$--tangle with paths. 
 Then $\tZ(M,K)$ is group-like and $\tZ(M,K)=\expd(W_\gamma)\sqcup H$ for some substantial and group-like $H$.
\end{lemma}

    \subsection{Application to $\Q$SK--pairs}

Let $(S,\kappa)$ be a $\Q$SK--pair. Let $M$ be the $\Q$--cube obtained from $S$ by removing the interior of a ball $B^3$ disjoint from $\kappa$. 
Isotoping $\kappa$ in $M$ and fixing a boundary parametrization $m$ of $M$, one can view $\kappa$ as the knot $\bK$ associated with a Lagrangian cobordism 
with one path $(M,K)$. Since the top and bottom words are empty, we get a Lagrangian $q$--cobordism with one path. 
\begin{proposition} \label{propinvariantQSK}
 Let $(S,\kappa)$ be a $\Q$SK--pair. Define as above an associated Lagrangian $q$--cobordism with one path $(M,K)$. Then $\tZ(S,\kappa)=\tZ(M,K)$ defines an invariant 
 of $\Q$SK--pairs, which coincides with the Kricker invariant $Z^\mathrm{rat}$ for knots in $\Z$--spheres.
\end{proposition}
\begin{proof}
 When associating a cobordism with one path with a $\Q$SK--pair, we make a choice in the way we isotope the knot to the closure of a path. Once we work 
 with a surgery presentation of our cobordism, this choice corresponds to the sweeping move represented in Figure~\ref{figsweepingmove}. 
\begin{figure}[htb]
\begin{center}
\begin{tikzpicture} 
\begin{scope} [scale=0.3]
 \draw (0,0) -- (10,0) -- (10,10) -- (0,10) -- (0,0);
 \foreach \y in {1,7.5} {\draw (1,\y) -- (9,\y) -- (9,\y+1.5) -- (1,\y+1.5) -- (1,\y);}
 \draw (5,1.75) node {$T'$} (5,8.25) node {$T$};
 \foreach \x in {2,3,4,7,8} {\draw (\x,2.5) -- (\x,7.5);}
 \draw (2.5,5) node {$\scriptstyle{\dots}$} (7.5,3.5) node {$\scriptstyle{\dots}$} (7.5,6.5) node {$\scriptstyle{\dots}$};
 \draw (5,5) node {$\bullet$} -- (10,5);
\end{scope}
 \draw (4,1.5) node {$\longleftrightarrow$};
\begin{scope} [xshift=5cm,scale=0.3]
 \draw (0,0) -- (10,0) -- (10,10) -- (0,10) -- (0,0);
 \draw (8.5,4) .. controls +(0,1) and +(0,-1) .. (4,5) .. controls +(0,1) and +(0,-1) .. (8.5,6); 
 \foreach \x in {2,3,7,8} {\draww{(\x,2.5) -- (\x,7.5);}}
 \draww{(4,2.5) .. controls +(0,2) and +(0,-1) .. (8.5,4) (8.5,6) .. controls +(0,1) and +(0,-2) .. (4,7.5);}
 \draw (2.5,5) node {$\scriptstyle{\dots}$} (7.5,3) node {$\scriptstyle{\dots}$} (7.5,7) node {$\scriptstyle{\dots}$};
 \foreach \y in {1,7.5} {\draw (1,\y) -- (9,\y) -- (9,\y+1.5) -- (1,\y+1.5) -- (1,\y);}
 \draw (5,1.75) node {$T'$} (5,8.25) node {$T$};
 \draw (5,5) node {$\bullet$} -- (10,5);
\end{scope}
\draw (9,1.5) node {$\sim$};
\begin{scope} [xshift=10cm,scale=0.3]
 \draw (0,0) -- (10,0) -- (10,10) -- (0,10) -- (0,0);
 \draw (1.5,4) .. controls +(0,1) and +(0,-1) .. (4,5) .. controls +(0,1) and +(0,-1) .. (1.5,6);
 \foreach \x in {2,3,7,8} {\draww{(\x,2.5) -- (\x,7.5);}}
 \draww{(4,2.5) .. controls +(0,2) and +(0,-1) .. (1.5,4) (1.5,6) .. controls +(0,1) and +(0,-2) .. (4,7.5);}
 \foreach \y in {1,7.5} {\draw (1,\y) -- (9,\y) -- (9,\y+1.5) -- (1,\y+1.5) -- (1,\y);}
 \draw (5,1.75) node {$T'$} (5,8.25) node {$T$};
 \draw (2.5,5) node {$\scriptstyle{\dots}$} (7.5,3.5) node {$\scriptstyle{\dots}$} (7.5,6.5) node {$\scriptstyle{\dots}$};
 \draw (5,5) node {$\bullet$} -- (10,5);
\end{scope}
\end{tikzpicture}
\end{center} \caption{A sweeping move.} \label{figsweepingmove}
\end{figure}
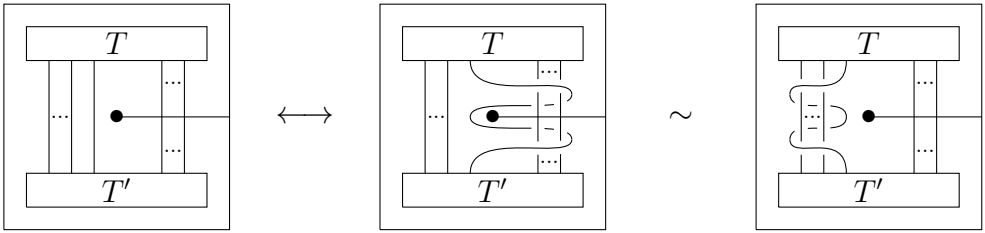
 But the right hand side diagram of this figure shows this move is trivial ---as noted in \cite[Lemma 3.26]{GK}. 
 
 Coincidence with $Z^\mathrm{rat}$ is direct by construction.
\end{proof}
\begin{remark}
 The above proof does not work for a cobordism with more than one path, so we do not get an invariant of boundary links in $\Q$--spheres. One may obtain 
 such an invariant by quotienting out the target space by suitable relations, see \cite{GK} for a construction of this kind.
\end{remark}

\begin{proposition}
 Let $(S_1,\kappa_1)$ and $(S_2,\kappa_2)$ be $\Q$SK--pairs. The invariant $\tZ$ is given on their connected sum by:  
 $$\tZ((S_1,\kappa_1)\sharp(S_2,\kappa_2))=\tZ(S_1,\kappa_1)\sqcup\tZ(S_2,\kappa_2).$$
\end{proposition}
\begin{proof}
 As previously, associate Lagrangian $q$--cobordisms with one path $(M_1,K_1)$ and $(M_2,K_2)$ with $(S_1,\kappa_1)$ and $(S_2,\kappa_2)$ respectively. 
 Construct a Lagrangian $q$--cobordism with one path $(M,K)$ associated with $(S,\kappa)=(S_1,\kappa_1)\sharp(S_2,\kappa_2)$ by stacking $(M_1,K_1)$ and $(M_2,K_2)$ 
 together in the $y$ direction. Now $(M_1,K_1)$ and $(M_2,K_2)$ are obtained from the cube $[-1,1]^3$ 
 with one disk by surgery on links $L_1$ and $L_2$ respectively. We obtain a surgery diagram for $(M,K)$ by drawing $L_1$ ``in front'' of $L_2$, or equivalently 
 ``around'' $L_2$, see Figure \ref{figstack}. The result follows from this latter diagram since there is no crossing between $L_1$ and $L_2$.
\begin{figure}[htb] 
\begin{center}
\begin{tikzpicture} [scale=0.5]
\begin{scope} 
 \draw (0,7.5) .. controls +(1,0) and +(0,1) .. (1,6) .. controls +(0,-2) and +(0,1) .. (-1,3);
 \draw (0,6) .. controls +(-1,0) and +(0,1) .. (-2,4) .. controls +(0,-4) and +(0,-3) .. (1,3);
 \draww{(0,7.5) .. controls +(-1,0) and +(0,1) .. (-1,6) .. controls +(0,-2) and +(0,1) .. (1,3);}
 \draww{(0,6) .. controls +(1,0) and +(0,1) .. (2,4) .. controls +(0,-4) and +(0,-3) .. (-1,3);}
 \draww{(0,3) arc (-90:90:2);}
 \draww{(0,2) arc (-90:90:3);}
 \draww{(0,8) .. controls  +(-4,0) and +(-1,0) .. (-2.5,4.5);}
 \draww{(0,3) .. controls  +(-3,0) and +(1,0) .. (-2.5,5.5);}
 \draww{(0,2) .. controls  +(-4,0) and +(-1,0) .. (-2.5,5.5);}
 \draww{(0,7) .. controls  +(-3,0) and +(1,0) .. (-2.5,4.5);}
 \draw (-4,0.5) rectangle (4,9.5);
 \draw (0,5) node {$\bullet$} -- (4,5);
 \draw[->] (2.24,7) -- (2.14,7.1) node[right] {$L_1$};
 \draw[->] (1.4,1.5) -- (1.45,1.6) node[right] {$L_2$};
\end{scope}
\draw (7,5) node {$\sim$};
\begin{scope} [xshift=14cm] 
\begin{scope} [scale=0.6,yshift=3.3cm] 
 \draw (0,7.5) .. controls +(1,0) and +(0,1) .. (1,6) .. controls +(0,-2) and +(0,1) .. (-1,3);
 \draw (0,6) .. controls +(-1,0) and +(0,1) .. (-2,4) .. controls +(0,-4) and +(0,-3) .. (1,3);
 \draww{(0,7.5) .. controls +(-1,0) and +(0,1) .. (-1,6) .. controls +(0,-2) and +(0,1) .. (1,3);}
 \draww{(0,6) .. controls +(1,0) and +(0,1) .. (2,4) .. controls +(0,-4) and +(0,-3) .. (-1,3);}
\end{scope}
\begin{scope} [xscale=1.1,yscale=1.3,xshift=0.2cm,yshift=-1.4cm] 
 \draww{(0,3) arc (-90:90:2);}
 \draww{(0,2) arc (-90:90:3);}
 \draww{(0,8) .. controls  +(-4,0) and +(-1,0) .. (-2.5,4.5);}
 \draww{(0,3) .. controls  +(-3,0) and +(1,0) .. (-2.5,5.5);}
 \draww{(0,2) .. controls  +(-4,0) and +(-1,0) .. (-2.5,5.5);}
 \draww{(0,7) .. controls  +(-3,0) and +(1,0) .. (-2.5,4.5);}
\end{scope}
 \draw (-4,0.5) rectangle (4,9.5);
 \draw (0,5) node {$\bullet$} -- (4,5);
 \draw[->] (2.5,7.5) -- (2.4,7.6) node[right] {$L_1$};
 \draw[->] (1.2,4.2) -- (1.2,4.3); \draw (1.05,4.3) node[right] {$L_2$};
\end{scope}
\end{tikzpicture}
\end{center}
\caption{Stacking diagrams.} \label{figstack}
\end{figure}
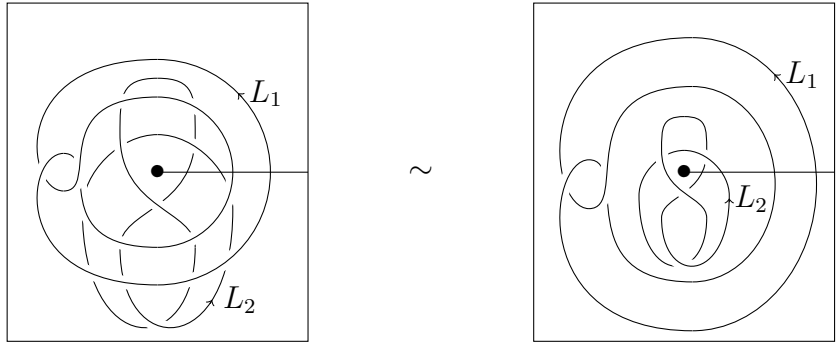
\end{proof}

      \section{Splitting formulas} \label{secformules}

We first mention useful lemmas, namely \cite[Lemma 4.3]{Mas} and \cite[Lemma 4.4]{Mas}. Recall the tensor $\mu(\chir)$ was defined in the introduction. 
\begin{lemma}[Massuyeau] \label{lemmaboundpara}
 For a $\Q$--handlebody $C$ of genus $g$, there exists a boundary parametrization $c:\partial C_0^g\to C$ such that $(C,c)\in\lcob(g,0)$.
\end{lemma}

\begin{lemma}[Massuyeau]
 Let $\chir=\left(\frac{C'}{C}\right)$ be an LP--pair of genus $g$. Take boundary parametrizations $c:\partial C_0^g\to C$ and $c':\partial C_0^g\to C'$ 
 compatible with the fixed identification $\partial C\cong\partial C'$ such that $(C,c)\in\lcob(g,0)$ and $(C',c')\in\lcob(g,0)$. Then: 
 $$\mu(\chir)=\tZ_1(C,c)-\tZ_1(C',c'),$$ where $\tZ_1$ is the i--degree 1 part of $\tZ$ and $\mu(\chir)$ is considered as an element of $\tA_{\Qt}(*_{\lfloor g\rceil^+})$ 
 {\em via} the inclusion $\Lambda^3 H_1(\totC;\Q)\hookrightarrow\tA_{\Qt}(*_{\lfloor g\rceil^+})$ defined by:
 $$[c_+(\beta_i)]\wedge[c_+(\beta_j)]\wedge[c_+(\beta_k)]\mapsto\rat{$i^+$}{$j^+$}{$k^+$}.$$
\end{lemma}

Let $(M,K)\in\tlcob_q(w,v)$. Let $\chir=(\chir_1,\dots,\chir_n)$ be a null LP--surgery on $(M,K)$. Let $e_i$ be the genus of $C_i$. 
For $1\leq i \leq n$, take boundary parametrizations $c_i:\partial C_0^{e_i}\to C_i$ and $c_i':\partial C_0^{e_i}\to C_i'$ compatible with the fixed identification 
$\partial C_i\cong\partial C_i'$ such that $(C_i,c_i)\in\lcob(e_i,0)$ and $(C_i',c_i')\in\lcob(e_i,0)$. Set $e=\sum_{i=1}^n e_i$. 
Take a collar neighborhood $m_-(F_f)\times[-1,\varepsilon-1]$ 
of the bottom surface $m_-(F_f)$. Take pairwise disjoint solid tubes $T_i$, $i=1,\dots,n$, such that $T_i$ connects $(c_i)_-(F_0)$ to a disk in 
$m_-(F_f)\times\{\varepsilon-1\}$ in the complement of the $C_j$'s, the collar neighborhood and $K$. This provides a decomposition of the cobordism $(M,K)$ as:
$$(M,K)=((C_1,\varnothing)\otimes\dots\otimes(C_n,\varnothing)\otimes Id_f)\circ(N,J),$$
where $f$ is the number of letters of $v$ (see Figure \ref{figdecompositionCob}). It is proved in \cite[Section 4.4]{Mas} that $N$ is a Lagrangian cobordism. The nullity condition on the surgery 
ensures that $\hat{J}$ is a boundary link. Thus $(N,J)$ is a Lagrangian cobordism with paths. 
\begin{figure}[htb]
\begin{center}
\begin{tikzpicture} [scale=0.25]
 \newcommand{\handleplus}[2]{
 \begin{scope} [xshift=#1cm,yshift=#2cm]
 \draw [white,line width=12pt] (11.5,2.5) arc (0:180:2);
 \draw[densely dashed,green] (11.5,2.5) circle (0.8 and 0.4) (7.5,2.5) circle (0.8 and 0.4);
 \draw (11.5,2) circle (0.5 and 0.25) (7.5,2) circle (0.5 and 0.25);
 \draw [white,line width=10pt] (11.5,2) arc (0:180:2);
 \draw[densely dashed] (11.5,2) circle (0.5 and 0.25) (7.5,2) circle (0.5 and 0.25);
 \draw (12,2) arc (0:180:2.5);
 \draw (11,2) arc (0:180:1.5);
 \draw[green] (12.3,2.5) .. controls +(0,3) and +(0,3) .. (6.7,2.5);
 \draw[green] (10.7,2.5) .. controls +(0,1) and +(0,1) .. (8.3,2.5);
 \draw[white, line width=2pt] (10.7,2.5) .. controls +(0,-0.5) and +(0,-0.5) .. (12.3,2.5) (6.7,2.5) .. controls +(0,-0.5) and +(0,-0.5) .. (8.3,2.5);
 \draw[green] (10.7,2.5) .. controls +(0,-0.5) and +(0,-0.5) .. (12.3,2.5) (6.7,2.5) .. controls +(0,-0.5) and +(0,-0.5) .. (8.3,2.5);
 \end{scope}}
 \newcommand{\handle}[2]{
 \begin{scope} [xshift=#1cm,yshift=#2cm]
 \draw (11.5,2) circle (0.5 and 0.25) (7.5,2) circle (0.5 and 0.25);
 \draw [white,line width=10pt] (11.5,2) arc (0:180:2);
 \draw[densely dashed] (11.5,2) circle (0.5 and 0.25) (7.5,2) circle (0.5 and 0.25);
 \draw (12,2) arc (0:180:2.5);
 \draw (11,2) arc (0:180:1.5);
 \end{scope}}
 \newcommand{\trou}{
 (2,0) ..controls +(0.5,-0.25) and +(-0.5,-0.25) .. (4,0)
 (2.3,-0.1) ..controls +(0.6,0.2) and +(-0.6,0.2) .. (3.7,-0.1)}
 \foreach \x in {8,30} {\draw (\x,4) -- (\x,19);}
 \draw[green] (30,4.5) -- (8,4.5) -- (0,0.5) (2.5,0.5) -- (2.5,0) -- (10.5,4) -- (10.5,4.5);
 \begin{scope} 
 \draw (0,0) -- (8,4) -- (30,4) -- (22,0) -- (0,0);
 \handleplus{4}{0}
 \handleplus{11}{0}
 \end{scope}
 \foreach \y in {5,10} {\draw[red,dashed] (30,\y+4) -- (19,\y+4);}
 \draw[red] (11,5) .. controls +(2,1) and +(2,-1) .. (18,7) .. controls +(-2,1) and +(-2,1) .. (19,9);
 \draw[red,->] (11.9,12) -- (11.8,11.9);
 \draw[red] (14,12) .. controls +(0,-2) and +(0,-2) .. (16,14) .. controls +(0,2) and +(0,2) .. (19,14);
 \draw[red,->] (14,5.72) -- (13.9,5.7);
 \draw[white,line width=3pt] (8,15) -- (22,15);
 \begin{scope} 
 \draw[yshift=15cm] (0,0) -- (8,4) -- (30,4) -- (22,0) -- (0,0);
 \handle{-0.5}{15}
 \handle{5.5}{15}
 \handle{11.5}{15}
 \end{scope}
 \foreach \x in {0,22} {\draw[white,line width=3pt] (\x,0.5) -- (\x,14.6); \draw (\x,0) -- (\x,15);}
 \begin{scope} [xshift=9cm,yshift=8.7cm,scale=0.7]
 \draw[blue] (0,0) ..controls +(0,1) and +(-2,1) .. (4,2);
 \draw[blue] (4,2) ..controls +(2,-1) and +(-2,-1) .. (8,2);
 \draw[blue] (8,2) ..controls +(2,1) and +(-1.2,0) .. (12,1.3);
 \draw[blue] (0,0) ..controls +(0,-1) and +(-2,-1) .. (4,-2);
 \draw[blue] (4,-2) ..controls +(2,1) and +(-2,1) .. (8,-2);
 \draw[white,line width=3pt] (8,-2) ..controls +(2,-1) and +(-1.2,0) .. (12,-1.3);
 \draw[blue] (8,-2) ..controls +(2,-1) and +(-1.2,0) .. (12,-1.3);
 \draw[white,line width=3pt] (14,2) ..controls +(2,1) and +(0,1) .. (18,0);
 \draw[blue] (14,2) ..controls +(2,1) and +(0,1) .. (18,0);
 \draw[blue] (14,-2) ..controls +(2,-1) and +(0,-1) .. (18,0);
 \draw[white,line width=3pt] (12,1.3) ..controls +(0.5,0) and +(-1,-0.5) .. (14,2);
 \draw[blue] (12,1.3) ..controls +(0.5,0) and +(-1,-0.5) .. (14,2);
 \draw[blue] (12,-1.3) ..controls +(0.5,0) and +(-1,0.5) .. (14,-2);
 \draw[yshift=0.08cm,blue] \trou;
 \draw[xshift=6cm,yshift=0.08cm,blue] \trou;
 \draw[xshift=12cm,yshift=0.08cm,blue] \trou;
 \end{scope} 
 \draw[red,dashed] (11,5) -- (22,5) -- (30,9);
 \draw[white,line width=3pt] (11,11) -- (11,10) -- (21,10);
 \draw[red] (11,10) .. controls +(0,2) and +(0,2) .. (14,12);
 \draw[red,dashed] (11,10) -- (22,10) -- (30,14);
 \draw[blue] (5.7,2.5) circle (0.3 and 0.15);
 \draw[white,line width=5pt] (5.7,2.5) .. controls +(0,2) and +(-1,-1) .. (9.7,7.7);
 \draw[blue] (5.4,2.5) .. controls +(0,2) and +(-1,-1) .. (9.4,7.9);
 \draw[blue] (6,2.5) .. controls +(0,2) and +(-1,-1) .. (9.95,7.45);
 \draw[white,line width=2pt] (10,4.25) -- (2.5,0.5);
 \draw[green] (10.5,4.5) -- (2.5,0.5);
 \draw[white,line width=2pt] (0.5,0.5) -- (2,0.5) (3,0.5) -- (6,0.5) (28,3.5) -- (29.5,4.25);
 \draw[green] (0,0.5) -- (22,0.5) -- (30,4.5) (4,2.5) -- (0,0.5);
 \draw (-5,8) node {$(M,K)=$};
 \draw (33,11.9) node {$\left.\resizebox{0cm}{1.96cm}{\phantom{rien}}\right\rbrace (N,J)$};
 \draw (1,-1.5) node {$C$};
 \draw (12,-1.5) node {$Id_2$};
\end{tikzpicture} \caption{Decomposition of a Lagrangian cobordism with paths} \label{figdecompositionCob}
\end{center}
\end{figure}
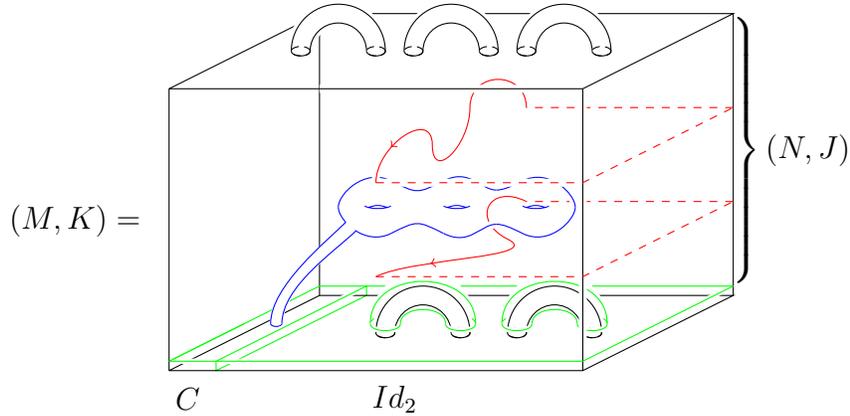

With the surgery $\chir$ is associated the tensor $\mu(\chir)\in\A_\Q(H_1(\totC;\Q))$. Let $W$ be a square matrix of size $e$ with coefficients in $\Qt$. 
Interpret $W$ as a hermitian form on $H_1(\totC;\Q)$ written in the basis $(([(c_i)_+(\beta_j)])_{1\leq j\leq e_i})_{1\leq i\leq n}$. 
Given an $H_1(\totC;\Q)$--colored Jacobi diagram, one can glue some legs of the diagram with $W$, see Figure \ref{figglue}. Changing the labels of the univalent vertices {\em via} the bijection 
$$[(c_i)_+(\beta_j)]\mapsto\sum_{\ell=1}^{i-1}e_\ell+j$$ 
onto $\{1,\dots,e\}$, this provides a diagram in $\tA_{\Qt}(*_{\lfloor e\rceil^+})$.

The following result is a direct adaptation of \cite[Section 4.4]{Mas}, with the winding matrices playing the role of the linking matrices.
\begin{proposition}
 Let $(M,K)$ be a Lagrangian $q$--cobordism with paths and let $(B,K,\gamma)$ be the associated bottom-top $q$--tangle with paths. Let $\chir=(\chir_1,\dots,\chir_n)$ 
 be a null LP--surgery on $(M,K)$. Define as above a decomposition of the cobordism $(M,K)$. Choose top and bottom words for $(N,J)$ and the $(C_i,\varnothing)$ 
 in order to get a decomposition of the Lagrangian $q$--cobordism $(M,K)$ as $(M,K)=((C_1,\varnothing)\otimes\dots\otimes(C_n,\varnothing)\otimes Id_v)\circ(N,J)$. 
 Let $(D,J,\varsigma)$ be the bottom-top $q$--tangle with paths associated with $(N,J)$. Let $\varsigma^\chirs$ be the subtangle of $\varsigma^-$ corresponding to the $C_i$'s. 
 Let $e$ be the number of components of $\varsigma^\chirs$. 
 Let $\tilde{\rho}_\chirs:\tA_{\Qt}(*_{\lfloor e\rceil^+})\to\tA_{\Qt}(*_{\lfloor g\rceil^+\cup\lfloor f\rceil^-})$ be the linear form which changes the labels 
 of the univalent vertices as follows: 
 $$\tilde{\rho}_\chirs(\ell^+)=\sum_{j=1}^g W_{\varsigma^\chirs}(\varsigma_j^+,\varsigma_\ell)\cdot j^+ +\sum_{i=1}^f W_{\varsigma^\chirs}(\varsigma_{e+i}^-,\varsigma_\ell)\cdot i^-.$$
 Then:
 $$\sum_{I\subset\{1,\dots,n\}}(-1)^{|I|}\tZ\left((M,K)(\chir_I)\right)\equiv_n
 \expd\left(\frac{1}{2}W_\gamma\right)\sqcup
 \tilde{\rho}_\chirs\left(\textrm{\begin{partext}{4.5cm} sum of all ways of gluing some legs of $\mu(\chir)$ with $W_{\varsigma^\chirs}/2$\end{partext}}\right),$$
 where $\chir_I=((\chir_i)_{i\in I})$ and $\equiv_n$ means ``equal up to i--degree at least $n+1$ terms''.
\end{proposition}
Note that $W_\gamma=W_{\varsigma\setminus\varsigma^\chirs}$ by Proposition \ref{propinvariancewind}.

For a cobordism with one path, the next result gives a more intrinsic version of these formulas, which does not refer to a decomposition of the cobordism. 
A similar result is given by in \cite[Lemma 4.1]{Mas} for a cobordism with no path.

Given a null LP--surgery $\chir=(\chir_1,\dots,\chir_n)$ on a Lagrangian cobordism with paths $(M,K)$, define a hermitian form 
$\ell_{(M,K)}(\chir):H_1(\totC;\Q)\times H_1(\totC;\Q)\to\Qt$ in the same way as $\ell_{(S,\kappa)}(\chir)$ was defined in the introduction. 
Also define a map $\rho_\chirs:\A_\Q(H_1(\totC;\Q))\to\tA_{\Qt}(*_{\lfloor g\rceil^+\cup\lfloor f\rceil^-})$ which changes the labels of the univalent vertices by first sending them in $H_1(M;\Q)$ {\em via} $H_1(\totC;\Q)\cong\oplus_{i=1}^nH_1(C_i;\Q)\to H_1(M;\Q)$, and then writing them in terms of the $[m_+(\beta_i)]$ and $[m_-(\alpha_i)]$. 

A direct adaptation of (the end of) \cite[Section 4.4]{Mas} gives:
\begin{proposition} \label{propformulas1}
 Let $(M,K)\in\tlcob_q(w,v)$ be a Lagrangian $q$--cobordism with one path. Let $\chir=(\chir_1,\dots,\chir_n)$ 
 be a null LP--surgery on $(M,K)$. Let $(B,K,\gamma)$ be the bottom-top tangle with one path associated with $(M,K)$. Then:
 $$\sum_{I\subset\{1,\dots,n\}}(-1)^{|I|}\tZ\left((M,K)(\chir_I)\right)\equiv_n
 \expd\left(\frac{1}{2}W_\gamma\right)\sqcup
 \rho_\chirs\left(\textrm{\begin{partext}{4.5cm} sum of all ways of gluing some legs of $\mu(\chir)$ with $\ell_{(M,K)}(\chir)/2$\end{partext}}\right).$$
\end{proposition}

\begin{proof}[Proof of Theorem \ref{thmain}]
 Use Propositions \ref{propinvariantQSK} and \ref{propformulas1}. The strut part disappears since we deal with a cobordism from $F_0$ to $F_0$. The map $\rho_\chirs$ 
 kills all terms with at least one univalent vertex since a Lagrangian cobordism from $F_0$ to $F_0$ has trivial first homology group over $\Q$.
\end{proof}

\def\cprime{$'$}
\providecommand{\bysame}{\leavevmode ---\ }
\providecommand{\og}{``}
\providecommand{\fg}{''}
\providecommand{\smfandname}{\&}
\providecommand{\smfedsname}{\'eds.}
\providecommand{\smfedname}{\'ed.}
\providecommand{\smfmastersthesisname}{M\'emoire}
\providecommand{\smfphdthesisname}{Th\`ese}

\end{document}